\documentclass[12pt]{article}
 
 \usepackage{amsmath,amssymb,amscd,amsthm,esint}
 
\usepackage{graphics,amsmath,amssymb,amsthm,mathrsfs}

\usepackage{graphics,amsmath,amssymb,amsthm,mathrsfs,amsfonts}

\usepackage{sidecap}
\usepackage{float}
\usepackage{extarrows}
\usepackage{booktabs}
\usepackage{verbatim}
\usepackage{hyperref}
\usepackage[usenames,dvipsnames]{xcolor}

\newtheorem{thm}{Theorem}[section]
\newtheorem{lemma}[thm]{Lemma}

\theoremstyle{definition}
\newtheorem{remark}[thm]{Remark}

\def\XXint#1#2#3{{\setbox0=\hbox{$#1{#2#3}{\int}$}
         \vcenter{\hbox{$#2#3$}}\kern-.5\wd0}}

\def\e{\varepsilon}

\numberwithin{equation}{section}

\begin{document}

\title{Homogenization of Parabolic Equations \\ with  Non-self-similar Scales}

\author{Jun Geng\thanks{Supported in part  by the NNSF of China (no. 11571152) and Fundamental Research Funds for the Central
Universities (lzujbky-2017-161).} \qquad
Zhongwei Shen\thanks{Supported in part by NSF grant DMS-1600520.}}

\date{}

\maketitle

\begin{abstract}

This paper is concerned with quantitative homogenization of second-order parabolic systems with periodic
coefficients varying rapidly in space and time, in different scales.
We obtain large-scale interior and boundary Lipschitz estimates as well as interior $C^{1, \alpha}$ and $C^{2, \alpha}$ estimates
by utilizing higher-order correctors.
We also  investigate the problem of convergence rates for initial-boundary value problems.

\medskip



\end{abstract}

\section{\bf Introduction}\label{section-1}

In this paper we shall be interested in quantitative homogenization of a parabolic operator
with periodic coefficients varying rapidly in space and time, in different scales.
More precisely, we consider the parabolic operator 
\begin{equation}\label{operator-0}
\partial_t +\mathcal{L}_\e
\end{equation}
in $\mathbb{R}^{d+1}$, where $\e>0$ and
\begin{equation}\label{operator-1}
\mathcal{L}_\e
=-\text{\rm div} \big( A(x/\e, t/\e^k)\nabla \big),
\end{equation}
with $0<k<\infty$.
We will assume that the coefficient matrix
$A=A(y,s)=\big(a_{ij}^{\alpha\beta} (y, s)\big)$, with $1\le i, j \le d$ and $1\le\alpha, \beta\le m$,
  is real, bounded measurable  and satisfies the ellipticity condition,
\begin{equation}\label{ellipticity}
\| A\|_\infty  \le \mu^{-1} \quad \text{ and }\quad
\mu |\xi|^2  \le  a^{\alpha\beta}_{ij}(y,s)\xi_i^\alpha
\xi^\beta_j
\end{equation}
 for any 
$\xi=(\xi_i^\alpha ) \in \mathbb{R}^{m\times d} \text{ and  a.e. }  (y,s)\in \mathbb{R}^{d+1}$,
where $\mu >0$ (the summation convention is used throughout).
We also assume that $A$ is 1-periodic in $(y, s)$; i.e.,  
\begin{equation}\label{periodicity}
A(y+z,s+t)=A(y,s)~~~\text{ for }(z,t)\in \mathbb{Z}^{d+1}\text{ and a.e. }(y,s)\in \mathbb{R}^{d+1}.
\end{equation}

The qualitative homogenization theory for the operator (\ref{operator-0}) 
has been known since  the  1970s (see e.g. \cite{BLP-2011}).
As $\e\to 0$, the weak solution  $u_\e$ of the initial-Dirichlet problem 
for the parabolic system $(\partial_t + \mathcal{L}_\e )u_\e=F$ in $\Omega_T =
\Omega \times (0, T)$
converges weakly in $L^2(0, T; H^1(\Omega))$
and strongly in $L^2(\Omega_T)$.
Moreover, the limit $u_0$ is the solution of the initial-Dirichlet problem for
$(\partial_t  +\mathcal{L}_0)  u_0=F$  in $\Omega_T$,
where $\mathcal{L}_0$ is a second-order elliptic operator  with constant coefficients.
Furthermore, the (homogenized) coefficients of $\mathcal{L}_0$ as well as  the first-order correctors depend on $k$,
but only for three separated cases: $0<k<2$; $k=2$; and $2<k<\infty$. 
For more recent work on multiscale convergence and reiterated homogenization,
see \cite{Allaire-1996, Holmbom, Floden-2006, Pastukhova-2009, Woukeng-2010} and references therein. 

In recent years there  is  a great amount of interest in the quantitative homogenization theory
for partial differential equations, where
one is concerned with problems  related to the large-scale  regularity and  convergence rates for solutions $u_\e$.
Major progress has been made for elliptic equations and systems in the periodic and non-periodic settings 
(see \cite{AL-1987, KLS-2013-N, Suslina-2013, Gloria-2015, Armstrong-2016, AM-2016, AKM-2016, Otto-2016,
Gloria-2017, Armstrong-2017, A-Shen-2016, Shen-book, Armstrong-book} and  references therein).
Some of these work has been extended to parabolic equations and systems in the self-similar  case $k=2$.
In particular, we established the  large-scale Lipschitz and $W^{1, p}$ estimates in \cite{Geng-Shen-2015} 
and studied the problem of  convergence rates  in $L^2(\Omega_T)$
as well as error estimates for two-scale expansions  in $L^2(0, T; H^1(\Omega))$  in \cite{Geng-Shen-2017}.
Also see related work in \cite{Suslina-M-2015, Xu-Zhou, Niu-Xu-2018, Niu-Xu-2019}.
Most recently,  in \cite{Geng-Shen-2018}
we study the asymptotic behavior of the  fundamental solution and its derivatives and
establish sharp estimates for the remainders.
We refer the reader to 
\cite{Armstrong-2018} for quantitative stochastic homogenization of parabolic equations.

If $k\neq 2$,  the $\e$ scaling in the coefficient matrix  $A(x/\e, t/\e^k)$  is not consistent with the intrinsic  scaling of 
the second-order parabolic equations.
To the authors' best knowledge,
very few quantitative results are known in this case, where
direct extensions of the techniques developed for elliptic equations fail.

In this paper we develop a new approach to  study homogenization of
parabolic equations and systems  with non-self-similar scales.
This allows us to establish large-scale interior and boundary 
Lipschitz estimates for  the parabolic operator (\ref{operator-0})
with any $0<k<\infty$, under conditions (\ref{ellipticity}) and (\ref{periodicity}).

Let $Q_r (x_0, t_0)=B(x_0, r) \times (t_0-r^2, t_0)$ denote a parabolic cylinder.
The following is one of the main results of the paper.

\begin{thm}\label{m-theorem-1}
Assume $A=A(y, s)$ satisfies (\ref{ellipticity}) and (\ref{periodicity}).
Let $u_\e$ be a weak solution to 
\begin{equation}
(\partial_t +\mathcal{L}_\e )u_\e =F \quad \text{ in } Q_R=Q_R (x_0, t_0),
\end{equation}
where $R>\e+\e^{k/2}$ and  $F\in L^p(Q_R)$ for some $p> d+2$.
Then for any $\e +\e^{k/2}\le r<R$,
\begin{equation}\label{L-estimate-1}
\left(\fint_{Q_r} |\nabla u_\e|^2\right)^{1/2}
\le C \left\{ \left(\fint_{Q_R} |\nabla u_\e|^2 \right)^{1/2} 
+R \left(\fint_{Q_R} |F|^p \right)^{1/p} \right\},
\end{equation}
where $C$ depends only on $d$, $m$, $p$, and $\mu$.
\end{thm}

The  inequality (\ref{L-estimate-1})  is a large-scale interior  Lipschitz estimate.
We also obtain large-scale $C^{1, \alpha}$ and $C^{2, \alpha}$ excess-decay estimates 
for solutions of $\partial_t +\mathcal{L}_\e$ (see Sections \ref{section-4} and \ref{section-5}).
Regarding the condition $R> r\ge \e+\e^{k/2}$,
we  point out  that there exists  $u_\e$ such that
$(\partial_t+\mathcal{L}_\e) u_\e =0$ in $\mathbb{R}^{d+1}$ and
$\nabla u_\e$ is $\e$-periodic in $x$ and $\e^k$-periodic in $t$ 
(the solution $u_\e$ is given by $x_j +\e \chi^\lambda_j (x/\e, t/\e^2)$
with $\lambda=\e^{k-2}$; see Section \ref{section-2}).
Note that if the periodic cell $(0, \e)^d \times (-\e^k, 0)$  for $\nabla u_\e$ is contained in
the parabolic cylinder $Q_r (0, 0)$, then $r^2\ge \e^k$ and $2r\ge \sqrt{d} \e$.
This implies that $r\ge (\e +\e^{k/2})/4$.
As a result, the condition $R> r\ge \e+\e^{k/2}$ for (\ref{L-estimate-1}) is more or less necessary.

The next theorem gives the large-scale boundary Lipschitz estimate.
Let $\Omega$ be a bounded $C^{1, \alpha}$ domain in $\mathbb{R}^d$ for some $\alpha>0$.
Define $D_r (x_0, t_0)=\big( B(x_0, r)\cap \Omega \big) \times (t_0-r^2, t_0)$ and
$\Delta_r (x_0, t_0)=\big( B(x_0, r)\cap \partial\Omega \big) \times (t_0-r^2, t_0)$,
where 
$x_0\in \partial\Omega$ and $t_0 \in \mathbb{R}$.

\begin{thm}\label{m-theorem-3}
Assume $A=A(y, s)$ satisfies (\ref{ellipticity}) and (\ref{periodicity}).
Suppose that 
$(\partial_t + \mathcal{L}_\e) u_\e =F$ in $D_R=D_R (x_0, t_0)$ and
$u_\e =f$ on $\Delta_R=D_R (x_0, t_0)$,
where $\e +\e^{k/2}<R\le 1$ and $F\in L^p (D_R)$ for some $p>d+2$.
Then for any $\e +\e^{k/2}  \le r< R$,
\begin{equation}\label{BL-estimate-1}
\left(\fint_{D_r} |\nabla u_\e|^2\right)^{1/2}
\le C \left\{ \left(\fint_{D_R} |\nabla u_\e|^2 \right)^{1/2} 
+R^{-1} \| f\|_{C^{1+\alpha}(\Delta_R)}
+R \left(\fint_{D_R} |F|^p \right)^{1/p} \right\},
\end{equation}
where $C$ depends only on $d$, $m$, $p$, $\mu$, and $\Omega$.
\end{thm}

In this paper we also investigate the rate of convergence in $L^2(\Omega_T)$ for the initial-Dirichlet problem,
\begin{equation}\label{DP-00}
(\partial_t +\mathcal{L}_\e )u_\e=F 
\quad \text{ in } \Omega_T \quad
\text{ and } \quad u_\e =f \quad \text{ on } \partial_p \Omega_T,
\end{equation}
where $\partial_p \Omega_T$ denotes the parabolic boundary of $\Omega_T$.

\begin{thm}\label{m-theorem-2}
Assume $A=A(y, s)$ satisfies (\ref{ellipticity}) and (\ref{periodicity}).
Also assume that $\|\partial_s A\|_\infty<\infty$ for $0<k<2$ and  $\|\nabla^2 A\|_\infty<\infty$
for $k>2$.
Let $u_\e$ be  the weak solution of (\ref{DP-00})  and $u_0$ the homogenized solution, where
$\Omega$ is a bounded $C^{1, 1}$ domain in $\mathbb{R}^d$ and $0<T<\infty$.
Then
\begin{equation}\label{rate-0}
\aligned
 & \| u_\e -u_0\|_{L^2(\Omega_T)}\\
 & 
\le  C \Big\{ \| u_0\|_{L^2(0, T; H^2(\Omega))}
+\|\partial_t u_0\|_{L^2(\Omega_T)} \Big\}
\cdot 
\left\{
\aligned
&  \e^{k/2} 
 &  \quad &  \text{if } \ 0<k\le 4/3,\\
&\e^{2-k} 
&  \quad &  \text{if } \ 4/3< k<  2,\\
& \e^{k-2} & \quad & \text{if } \ 2<k<3,\\
&  \e
& \quad &  \text{if } \ k=2 \text{ or} \ 3\le  k<\infty ,\\
\endaligned
\right.
\endaligned
\end{equation}
for any $0<\e<1$,
where $C$ depends only on $d$, $m$, $k$, $A$, $\Omega$, and $T$.
\end{thm}

We now describe our general approach to Theorems \ref{m-theorem-1}, \ref{m-theorem-3}, and \ref{m-theorem-2}.
The key insight is to introduce a new scale $\lambda\in (0, \infty)$ and consider the operator 
\begin{equation}\label{op-l}
\mathcal{L}_{\e, \lambda} =-\text{\rm div} \big( A_\lambda (x/\e, t/\e^2)\nabla  \big),
\end{equation}
where $A_\lambda (y, s) = A(y, s/\lambda)$.
Observe that the coefficient matrix $A_\lambda$ is 1-periodic in $y$ and $\lambda$-periodic in $s$.
Moreover, for each $\lambda$ fixed, the scaling of the parameter  $\e$
in $A_\lambda (x/\e, t/\e^2)$ is 
consistent with the intrinsic  scaling of the second-order parabolic operator $\partial_t +\mathcal{L}_{\e, \lambda}$.
As a result, we may extend some of recently developed techniques for 
elliptic equations to the parabolic equation $(\partial_t +\mathcal{L}_{\e, \lambda}) u_{\e, \lambda}  =F$, as in the case $k=2$.
We point out that for the results to be useful, it is crucial that
the bounding constants $C$ in the estimates  of solutions $u_{\e, \lambda}$ do not depend on $\lambda$ (and $\e$).
This allows to use the observation  $\mathcal{L}_\e = \mathcal{L}_{\e, \lambda}$ for $\lambda=\e^{k-2}$
and prove Theorems \ref{m-theorem-1} and \ref{m-theorem-3}.
The approach also  leads to large-scale $C^{1, \alpha}$ and $C^{2, \alpha}$ excess-decay estimates
as well as a Liouville property,
expressed in terms of correctors for $\partial_t +\mathcal{L}_{\e, \lambda}$.

The  approach described above works equally well for  the problem of convergence rates.
In addition to the observation $\mathcal{L}_{\e, \lambda}=\mathcal{L}_\e$ for $\lambda=\e^{k-2}$,
we also use the fact that as $\lambda\to \infty$,
the homogenized coefficient matrix $\widehat{A_\lambda}$  for  $\partial_t+\mathcal{L}_{\e, \lambda}$
converges to $\widehat{A_\infty}$, the homogenized coefficient matrix  for $\partial_t + \mathcal{L}_\e$ in the case
$0<k<2$.  If $\lambda\to 0$, then $\widehat{A_\lambda} \to \widehat{A_0}$,
the homogenized coefficient matrix for $\partial_t + \mathcal{L}_\e$ in the case $2<k<\infty$.

The paper is organized as follows.
In Section \ref{section-2} we introduce the first-order correctors $\chi^\lambda$ and homogenized coefficients 
for  $\partial_t +\mathcal{L}_{\e, \lambda}$, with $\lambda>0$ fixed,
as well as correctors and homogenized coefficients for $\mathcal{L}_\e$ in (\ref{operator-0}) with $0<k<\infty$.
We also establish estimates of  $|\widehat{A_\lambda} -\widehat{A_\infty}|$  for $\lambda>1$,  and of 
$|\widehat{A}_\lambda -\widehat{A_0}|$ for $0< \lambda<1$, under additional regularity assumptions on $A$.
These estimates are used in the proof of Theorem \ref{m-theorem-2}.
In Section \ref{section-3} we prove an approximation result for solutions of $(\partial_t +\mathcal{L}_{\e, \lambda}) u_{\e, \lambda}=F$
in a parabolic cylinder.
This is done by using $\e$-smoothing and dual correctors.
The proof follows  the approach used in \cite{Geng-Shen-2017} by the  present authors  for the case $\lambda=1$.
The proof of Theorem \ref{m-theorem-1} is given in Section \ref{section-4}, where we also
establish a large-scale  $C^{1, \alpha}$ estimate.
In Section \ref{section-5} we introduce  second-order correctors for  the operator $\partial_t +\mathcal{L}_{\e, \lambda}$ and
prove a large-scale $C^{2, \alpha}$ estimate. 
The large-scale boundary Lipschitz estimate in Theorem \ref{m-theorem-3}
is proved in Section \ref{section-7}.
We remark that the approaches used in Sections \ref{section-4}, \ref{section-5}, and \ref{section-7} are motivated 
by recently developed techniques for studying the  large-scale regularity in the homogenization theory 
for elliptic equations and systems \cite{Gloria-2015, Armstrong-2016, AM-2016, AKM-2016, Otto-2016,
Gloria-2017, Armstrong-2017, A-Shen-2016}.
Finally, we give the proof of Theorem \ref{m-theorem-2} in Section \ref{section-6},
where we also obtain error estimates for a two-scale expansion in $L^2(0, T; H^1(\Omega))$.

The summation convention is used throughout.
We will use $\fint_E u$ to denote the $L^1$ average of $u$ over the set $E$; i.e. 
$\fint_E u  =\frac{1}{|E|} \int_E u$.
For notational simplicity we will assume $m=1$ in the rest of the paper.
However, no particular fact pertain to the scalar case is ever used.
All results and proofs extend readily  to the case $m>1$ - the case of parabolic  systems.



\section{\bf Correctors and homogenized coefficients}\label{section-2}

Let $A=A(y, s)$ be a matrix satisfying conditions (\ref{ellipticity}) and (\ref{periodicity}).
For $\lambda>0$, define
\begin{equation}\label{A-Lambda}
A_\lambda =A_\lambda (y, s)=A(y, s/\lambda) \quad \text{ for } (y, s)\in \mathbb{R}^{d+1}.
\end{equation}
The matrix $A_\lambda$ is $(1, \lambda)$-periodic in $(y, s)$; i.e., 
$$
A_\lambda (y+z, s+\lambda t)=A_\lambda (y, s)\quad \text{ for } (z, t)\in \mathbb{Z}^{d+1}.
$$
Let $\chi^\lambda=\chi^\lambda (y, s)=(\chi_1^\lambda (y, s), \dots, \chi_d^\lambda (y, s))$, where
$\chi_j^\lambda=\chi_j^\lambda (y, s)$ is the weak solution of the parabolic cell problem:
\begin{equation}\label{cell-lambda}
\left\{
\aligned
 & \partial_s \chi_j^\lambda -\text{\rm div} \big( A_\lambda \nabla \chi_j^\lambda\big)
=\text{\rm div}\big( A_\lambda \nabla y_j\big) \quad \text{ in } \mathbb{R}^{d+1},\\
& \chi_j^\lambda \text{ is } (1, \lambda) \text{-periodic in } (y, s),\\
& \int_0^\lambda \!\!\!\int_{\mathbb{T}^d} \chi_j^\lambda  (y, s)\, dy ds=0,
\endaligned
\right.
\end{equation}
where $\mathbb{T}^d=[0, 1)^d=\mathbb{R}^d /\mathbb{Z}^d$.
By the energy estimates,
\begin{equation}\label{2.1}
\fint_0^\lambda  \!\!\! \int_{\mathbb{T}^d} |\nabla \chi_j^\lambda|^2\, dyds
\le C,
\end{equation}
where $C$ depends only on $d$ and $\mu$.
Since
$$
\partial_s \int_{\mathbb{T}^d} \chi_j^\lambda (y, s)\, dy=0,
$$
we obtain, by the integral  condition in (\ref{cell-lambda}),
\begin{equation}\label{2.2}
\int_{\mathbb{T}^d} \chi_j^\lambda (y, s) \, dy=0.
\end{equation}
This, together with (\ref{2.1}) and Poincar\'e's inequality, gives
\begin{equation}\label{2.3}
 \fint_0^\lambda \!\!\! \int_{\mathbb{T}^d} |\chi_j^\lambda|^2\, dyds
\le C,
\end{equation}
where $C$ depends only on $d$ and $\mu$.
Since $\chi^\lambda$ and $\nabla \chi^\lambda$ are $(1, \lambda)$-periodic in $(y, s)$,
it follows from (\ref{2.1}) and (\ref{2.3}) that if $r\ge 1+\sqrt{\lambda}$, 
\begin{equation}\label{Q-e-1}
\left(\fint_{Q_r}  \left(|\nabla \chi^\lambda|^2 +|\chi^\lambda|^2 \right) \right)^{1/2}
\le C
\end{equation}
for any $Q_r =Q_r (x, t)$, where $C$ depends only on $d$ and $\mu$.

Let
\begin{equation}\label{2.4}
\widehat{A_\lambda}
=\fint_0^\lambda \!\!\! \int_{\mathbb{T}^d}
\left( A_\lambda + A_\lambda \nabla \chi^\lambda \right) dyds.
\end{equation}

\begin{lemma}\label{lemma-2.1}
There exists $C>0$, depending only on $d$ and $\mu$, such that
$|\widehat{A_\lambda}|\le C$. Moreover,  
\begin{equation}\label{2.5}
\mu |\xi|^2 \le \xi \cdot \widehat{A_\lambda} \xi
\end{equation}
for any $\xi \in \mathbb{R}^d$.
\end{lemma}

\begin{proof}
The inequality $|\widehat{A_\lambda}|\le C$ follows readily from (\ref{2.1}).
To see (\ref{2.5}),
we note that
$$
\aligned
\xi \cdot \widehat{A_\lambda} \xi
&=\fint_0^\lambda \!\!\!\int_{\mathbb{T}^d}
A_\lambda \nabla ( \xi\cdot y +\xi \cdot \chi^\lambda)
\cdot \nabla ( \xi\cdot y + \xi \cdot \chi^\lambda)\, dy ds\\
& \ge \mu \fint_0^\lambda \!\!\!\int_{\mathbb{T}^d}
|\nabla (\xi \cdot y + \xi \cdot \chi^\lambda)|^2\, dyds\\
& =\mu \fint_0^\lambda \!\!\!\int_{\mathbb{T}^d}
( |\xi|^2 + |\xi\nabla \chi^\lambda|^2) \, dy ds\\
& \ge \mu |\xi|^2
\endaligned
$$
for any $\xi \in \mathbb{R}^d$, where we have used the fact 
$\int_0^\lambda \!\!\int_{\mathbb{T}^d} \nabla\chi^\lambda\, dy ds=0$.
\end{proof}

It is well known that for a fixed $\lambda>0$,
the homogenized operator for the parabolic operator 
\begin{equation}
\partial_t +\mathcal{L}_{\e, \lambda}=
\partial_t -\text{\rm div} \big(A_\lambda (x/\e, t/\e^2)\nabla \big)
\end{equation}
is given by $\partial_t -\text{\rm div} \big( \widehat{A_\lambda}\nabla \big)$ \cite{BLP-2011}.
In particular, if $k=2$, the homogenized operator for the  operator in (\ref{operator-0}) is
given by  $\partial_t -\text{\rm div} \big( \widehat{A_\lambda}\nabla \big)$ with $\lambda=1$.

To introduce the homogenized operator for $\partial_t +\mathcal{L}_\e$  in (\ref{operator-0}) for 
$k\neq 2$, we first consider the case $0<k<2$.
Let $\chi^\infty =\chi^\infty (y, s)= (\chi_1^\infty (y, s), \dots, \chi_d^\infty (y, s) )$,
where $\chi_j^\infty=\chi_j^\infty (y, s)$ denotes the weak solution of the (elliptic) cell problem,
\begin{equation}\label{2.8}
\left\{
\aligned
& -\text{\rm div} \big( A\nabla \chi_j^\infty) =\text{\rm div}(A\nabla y_j) \quad \text{ in } \mathbb{R}^{d+1},\\
&\chi_j^\infty \text{ is 1-periodic in } (y, s),\\
& \int_{\mathbb{T}^d} \chi_j^\infty (y, s)\, dy=0.
\endaligned
\right.
\end{equation}
By the energy estimates and Poincar\'e's inequality,
\begin{equation}\label{2.10}
\int_{\mathbb{T}^d} \left( |\nabla \chi_j^\infty(y, s)|^2 +|\chi_j^\infty(y, s)|^2\right)dy \le C,
\end{equation}
for a.e. $s\in \mathbb{R}$, where $C$ depends only on $d$ and $\mu$.
Let
\begin{equation}\label{2.11}
\widehat{A_\infty}
=\int_0^1 \!\!\! \int_{\mathbb{T}^d}
\left( A + A \nabla \chi^\infty \right) dyds.
\end{equation}
It follows from (\ref{2.10}) that $|\widehat{A_\infty}|\le C$, where $C$ depends only on $d$ and $\mu$.
By the same argument as in the proof of Lemma \ref{lemma-2.1}, one may also show that
\begin{equation}\label{2.12}
\mu |\xi|^2 \le \xi \cdot \widehat{A_\infty} \xi
\end{equation}
for any $\xi \in \mathbb{R}^d$.
For $0<k<2$, the homogenized operator for the parabolic operator in (\ref{operator-0}) is given by
$\partial_t -\text{\rm div} \big(\widehat{A_\infty}\nabla \big)$ (see \cite{BLP-2011}).

Next, we consider the case $2<k<\infty$.
Define
\begin{equation}\label{2.13}
\overline{A}=\overline{A}(y)
=\int_0^1 A(y, s)\, ds.
\end{equation}
Let $\chi^0 =\chi^0 (y)=( \chi^0_1 (y), \dots, \chi_d^0 (y) )$,
where $\chi_j^0=\chi_j^0 (y)$  is  the weak solution of the (elliptic) cell problem,
\begin{equation}\label{2.14}
\left\{
\aligned
& -\text{\rm div} \left( \overline{A}\nabla \chi_j^0 \right)=\text{\rm div} \left(\overline{A}\nabla y_j\right) \quad \text{ in } \mathbb{R}^d,\\
& \chi_j^0 \text{ is 1-periodic in } y,\\
&\int_{\mathbb{T}^d} \chi_j^0\, dy =0.
\endaligned
\right.
\end{equation}
As in the case $0<k<2$, by the energy estimates and Poincar\'e's  inequality,
\begin{equation}\label{2.15}
\int_{\mathbb{T}^d} \left( |\nabla \chi_j^0(y)|^2 +|\chi_j^0(y)|^2\right)dy \le C,
\end{equation}
where $C$ depends only on $d$ and $\mu$.
Let
\begin{equation}\label{2.16}
\widehat{A_0}
=\int_0^1 \!\!\! \int_{\mathbb{T}^d}
\left( A + A \nabla \chi^0 \right) dyds
= \int_{\mathbb{T}^d}
\left( \overline{A} + \overline{A} \nabla \chi^0 \right) dy.
\end{equation}
It follows from (\ref{2.15}) that $|\widehat{A_0}|\le C$, where $C$ depends only on $d$ and $\mu$.
By the same argument as in the proof of Lemma \ref{lemma-2.1}, we obtain 
\begin{equation}\label{2.19}
\mu |\xi|^2 \le \xi \cdot \widehat{A_0} \xi
\end{equation}
for any $\xi \in \mathbb{R}^d$.
For $2<k<\infty$, the homogenized operator for $\partial_t +\mathcal{L}_\e$ in (\ref{operator-0}) is given by
$\partial_t -\text{\rm div} \big(\widehat{A_0}\nabla \big)$ (see \cite{BLP-2011}).

In the remaining of this section we study the asymptotic behavior of 
the matrix $\widehat{A_\lambda}$, as $\lambda\to \infty$ and as $\lambda \to 0$.
We begin with a lemma on the higher integrability of  $\nabla\chi^\lambda$.

\begin{lemma}\label{lemma-2.2}
Let $\chi^\lambda$ be defined by (\ref{cell-lambda}).
Then there exists $q>2$, depending on $d$ and $\mu$, such that
\begin{equation}\label{2.00}
\left(\fint_0^\lambda \!\!\!\int_{\mathbb{T}^d}
|\nabla \chi^\lambda|^q\, dy ds \right)^{1/q}
\le C,
\end{equation}
where $C$ depends only on $d$ and $\mu$.
\end{lemma}

\begin{proof}
Let $u(y, s)=y_j +\chi_j^\lambda$.
Then
$
\partial_s u -\text{\rm div} ( A_\lambda \nabla u) =0   \text{ in } \mathbb{R}^{d+1}.
$
By  Meyers-type estimates for parabolic systems (see e.g. \cite[Appendix]{Armstrong-2018}),
 there exist $q>2$ and $C>0$, depending only on $d$ and $\mu$, such that
\begin{equation}\label{2.01}
\left(\fint_{Q_r} |\nabla u|^q\, dyds \right)^{1/q}
\le C \left(\fint_{Q_{2r}} |\nabla u|^2\, dyds \right)^{1/2}
\end{equation}
for any $Q_r =Q_r (x, t) =B(x, r) \times (t-r^2, t)$.
It follows that
\begin{equation}\label{2.02}
\left(\fint_{Q_r} |\nabla \chi_j^\lambda |^q\, dyds \right)^{1/q}
\le  C+ C \left(\fint_{Q_{2r}} |\nabla \chi_j^\lambda |^2\, dyds \right)^{1/2}.
\end{equation}
Choose $r>1+\sqrt{\lambda}$ so large that $\mathbb{T}^d \times (0, \lambda)\subset Q_r$.
Since $\nabla \chi_j^\lambda$ is 1-periodic in $y$ and $\lambda$-periodic in $s$, we obtain 
$$
\aligned
\left(\fint_0^\lambda\!\!\!   \int_{\mathbb{T}^d}
|\nabla \chi_j^\lambda|^q \, dyds \right)^{1/q}
& \le C \left(\fint_{Q_r} 
|\nabla \chi_j^\lambda|^q \, dyds \right)^{1/q}\\
& \le C + C \left(\fint_{Q_{2r}}
|\nabla \chi_j^\lambda |^2\, dyds \right)^{1/2}\\
&\le C,
\endaligned
$$
where we have used (\ref{Q-e-1}) for the last step.
\end{proof}

\begin{thm}\label{theorem-2.1}
Assume  $A=A(y, s)$ satisfies conditions (\ref{ellipticity}) and (\ref{periodicity}).
Then 
\begin{equation}\label{2.20}
\widehat{A_\lambda} \to \widehat{A_\infty} \quad \text{  as }\lambda \to \infty.
\end{equation}
Moreover, if $\|\partial_s A\|_\infty<\infty$, then
\begin{equation}\label{2.21}
|\widehat{A_\lambda} -\widehat{A_\infty}|\le C \lambda^{-1} \|\partial_s A\|_\infty
\end{equation}
for any $\lambda>1$,
where $C$ depends only on $d$ and  $\mu$.
\end{thm}

\begin{proof}
We first prove (\ref{2.21}).
Observe that
$$
\widehat{A_\lambda}-\widehat{A_\infty}
=\int_0^1\!\!\! \int_{\mathbb{T}^d}
A (y, s) \nabla \left\{ \chi^\lambda (y, \lambda s)-\chi^\infty (y, s) \right\} dyds.
$$
It follows by the Cauchy inequality that
\begin{equation}\label{2.22}
| \widehat{A_\lambda}-\widehat{A_\infty}|
\le C \left(  \int_0^1\!\!\! \int_{\mathbb{T}^d}
| \nabla \left\{ \chi^\lambda (y, \lambda s)-\chi^\infty (y, s) \right\} |^2 \, dyds\right)^{1/2}.
\end{equation}
By the definitions of $\chi^\lambda$ and $ \chi^\infty$,
$$
\frac{1}{\lambda}
\frac{\partial}{\partial s}
\big\{ \chi_j^\lambda (y, \lambda s) \big\}
-\text{\rm div} \big\{ A(y, s) \nabla \big( \chi_j^\lambda (y, \lambda s) -\chi_j^\infty (y, s) \big) \big\}=0
\quad \text{ in } \mathbb{T}^{d+1}.
$$
This leads to
$$
\aligned
 &\int_0^1\!\!\! \int_{\mathbb{T}^d}
A(y, s) \nabla \big\{ \chi_j^\lambda (y, \lambda s) -\chi_j ^\infty (y, s) \big\} 
\cdot \nabla \big\{ \chi_j^\lambda (y, \lambda s) -\chi_j^\infty  (y, s) \big\}   \, dyds\\
&=-\frac{1}{\lambda }
\int_0^1\!\!\! \int_{\mathbb{T}^d} 
\frac{\partial}{\partial s} 
\big\{  \chi_j^\lambda (y, \lambda s) \big\} \cdot \big\{ \chi_j^\lambda (y, \lambda s) -\chi_j ^\infty (y, s) \big\}  \, dyds\\
&
=-\frac{1}{\lambda }
\int_0^1\!\!\! \int_{\mathbb{T}^d} 
\frac{\partial}{\partial s} 
\big\{  \chi_j^\infty (y, s) \big\} \cdot \big\{ \chi_j^\lambda (y, \lambda s) -\chi^\infty _j (y, s) \big\}  \, dyds,
\endaligned
$$
where we have used the fact
$$
\int_0^1\!\!\! \int_{\mathbb{T}^d} 
\frac{\partial}{\partial s} 
\big\{  \chi_j^\lambda (y, \lambda s) -\chi_j^\infty(y, s)  \big\} \cdot \big\{ \chi_j^\lambda (y, \lambda s) -\chi_j^\infty  (y, s) \big\} \, dyds =0
$$
for the last step.
Hence, by (\ref{ellipticity}) and the Cauchy inequality,
$$
\aligned
& \mu \int_0^1\!\!\!\int_{\mathbb{T}^d}
|\nabla \big\{ \chi_j^\lambda (y, \lambda s) -\chi_j^\infty (y, s) \big\} |^2\, dyds\\
& \le \frac{1}{\lambda}
\left(\int_0^1\!\!\! \int_{\mathbb{T}^d} 
|\chi_j^\lambda(y, \lambda s) -\chi_j^\infty(y, s)|^2\, dyds \right)^{1/2}
\left(\int_0^1\!\!\! \int_{\mathbb{T}^d}  |\partial_s \chi_j^\infty(y, s)|^2\, dyds\right)^{1/2}.
\endaligned
$$
Since
$$
\int_{\mathbb{T}^d}
\chi_j^\lambda (y, \lambda s)\, dy 
=\int_{\mathbb{T}^d} \chi_j^\infty (y, s)\, dy=0,
$$
by Poincar\'e's inequality,
 we obtain
$$
\left( \int_0^1\!\!\!\int_{\mathbb{T}^d}
|\nabla \big\{ \chi_j^\lambda (y, \lambda s) -\chi_j^\infty (y, s) \big\} |^2\, dyds\right)^{1/2}
\le \frac{C}{\lambda}
\left(\int_0^1\!\!\! \int_{\mathbb{T}^d}  |\partial_s \chi_j^\infty(y, s)|^2\, dyds\right)^{1/2}.
$$
In view of (\ref{2.22}) we have proved that
\begin{equation}\label{2.23}
|\widehat{A_\lambda} -\widehat{A_\infty}|
\le \frac{C}{\lambda}
\left(\int_0^1\!\!\! \int_{\mathbb{T}^d}  |\partial_s \chi^\infty(y, s)|^2\, dyds\right)^{1/2},
\end{equation}
where $C$ depends only on $d$ and $\mu$.

To bound the right-hand side of (\ref{2.23}), we differentiate in $s$ the elliptic equation for 
$\chi_j^\infty$ to obtain
$$
-\text{\rm div} \big(A \nabla \partial_s \chi_j^\infty)
=\text{\rm div} \big(\partial_s A \nabla y_j ) +\text{\rm div} \big(\partial_s A \nabla \chi_j^\infty).
$$
It follows that
$$
\int_{\mathbb{T}^d} 
|\nabla \partial_s \chi_j^\infty (y, s)|^2\, dy
\le C \int_{\mathbb{T}^d} |\partial_s A(y, s)|^2\, dy
+ C \int_{\mathbb{T}^d} |\partial_s A(y, s)|^2 |\nabla \chi_j^\infty (y, s)|^2\, dy.
$$
By Meyers estimates, there exists some $q>2$, depending only on $d$ and $\mu$, such that
$$
\int_{\mathbb{T}^d} |\nabla \chi_j^\infty(y, s) |^q \, dy\le C,
$$
where $C$ depends only on $d$ and $\mu$.
Thus, by H\"older's inequality,
$$
\left(\int_0^1\!\!\!\int_{\mathbb{T}^d} |\nabla \partial_s \chi_j^\infty |^2\, dyds \right)^{1/2}
\le C \left(\int_0^1\!\!\!\int_{\mathbb{T}^d} |\partial_s A|^{p_0}\, dy ds \right)^{1/p_0},
$$
for $p_0 =\frac{2 q}{q-2}$.
In view of (\ref{2.23})  this gives 
\begin{equation}\label{2.24}
|\widehat{A_\lambda} -\widehat{A_\infty}|
\le \frac{C}{\lambda}
\left(\int_0^1\!\!\! \int_{\mathbb{T}^d}  |\partial_s A |^{p_0} \, dyds\right)^{1/p_0},
\end{equation}
by using Poincar\'e's  inequality.
As a consequence, we obtain (\ref{2.21}).

Finally, to prove (\ref{2.20}),
we let $D$ be a matrix satisfying conditions (\ref{ellipticity}) and (\ref{periodicity}).
Also assume that $D$ is smooth in $(y, s)$.
Let $\widehat{D_\lambda}$ and $\widehat{D_\infty}$ be defined in the same manner as
$\widehat{A_\lambda}$ and $\widehat{A_\infty}$, respectively.
By using the energy estimates as well as (\ref{2.00}),
it is not hard to show that
$$
|\widehat{A_\lambda}
-\widehat{D_\lambda} |
\le C \left(\int_0^1\!\!\!\int_{\mathbb{T}^d} 
|A -D|^{p_0} \, dyds \right)^{1/p_0},
$$
where $C$ depends only on $d$ and $\mu$.
A similar argument also gives
$$
|\widehat{A_\infty}
-\widehat{D_\infty} |
\le C \left(\int_0^1\!\!\!\int_{\mathbb{T}^d} 
|A -D|^{p_0} \, dyds \right)^{1/p_0}.
$$
Thus, by applying the estimate (\ref{2.24}) to the matrix $D$, we obtain 
$$
\aligned
|\widehat{A_\lambda} -\widehat{A_\infty} |
&\le |\widehat{A_\lambda} -\widehat{D_\lambda}|
+| \widehat{D_\lambda} -\widehat{D_\infty}|
+| \widehat{D_\infty} -\widehat{A_\infty}|\\
& \le C \left(\int_0^1\!\!\!\int_{\mathbb{T}^d} 
|A -D|^{p_0} \, dyds \right)^{1/p_0}
+ \frac{C}{\lambda}
\left(\int_0^1\!\!\! \int_{\mathbb{T}^d}  |\partial_s D |^{p_0} \, dyds\right)^{1/p_0}.
\endaligned
$$
It follows that
$$
\limsup_{\lambda\to \infty} 
|\widehat{A_\lambda} -\widehat{A_\infty} |
\le 
 C \left(\int_0^1\!\!\!\int_{\mathbb{T}^d} 
|A -D|^{p_0} \, dyds \right)^{1/p_0}.
$$
Since $p_0 =\frac{2q}{q-2}  <\infty$,
by using convolution,
we may approximate $A$ in $L^{p_0} (\mathbb{T}^{d+1})$ by a sequence of smooth matrices satisfying (\ref{ellipticity}) and (\ref{periodicity}).
As a result, we conclude  that $\widehat{A_\lambda} \to \widehat{A_\infty}$ as $\lambda \to \infty$.
\end{proof}

\begin{remark}
It follows from the proof of Theorem \ref{theorem-2.1} that
$$
\aligned
& \left(\fint_0^\lambda \!\!\!\int_{\mathbb{T}^d}
|\nabla \chi^\lambda (y, s) -
\nabla \chi^\infty (y, s/\lambda)|^2\, dyds \right)^{1/2}
+
\left(\fint_0^\lambda \!\!\!\int_{\mathbb{T}^d}
| \chi^\lambda (y, s) -
 \chi^\infty (y, s/\lambda)|^2\, dyds \right)^{1/2}\\
& \qquad\quad
\le C \lambda^{-1} \|\partial_s A\|_\infty.
\endaligned
$$
By the periodicity this  implies that if $ r \ge (1+\sqrt{\lambda} ) \e$,
\begin{equation}\label{Q-e-5}
\aligned
& \left(\fint_{Q_r}
|\nabla \chi^\lambda (y/\e, s/\e^2) -
\nabla \chi^\infty (y/\e, s/(\lambda \e^2)) |^2\, dyds \right)^{1/2}\\
& \qquad \qquad \qquad+
\left(\fint_{Q_r}
| \chi^\lambda (y/\e, s/\e^2) -
 \chi^\infty (y/\e, s/(\lambda\e^2)) |^2\, dyds \right)^{1/2}\\
& \qquad\quad
\le C \lambda^{-1} \|\partial_s A\|_\infty.
\endaligned
\end{equation}

\end{remark}

The next theorem is concerned with the limit of $\widehat{A_\lambda}$ as $\lambda\to 0$.

\begin{thm}\label{theorem-2.2}
Assume $A=A(y, s)$ satisfies conditions (\ref{ellipticity}) and (\ref{periodicity}).
Then 
\begin{equation}\label{2.30}
\widehat{A_\lambda} \to \widehat{A_0} \quad \text{  as }\lambda \to 0.
\end{equation}
Moreover, if $\|\nabla^2 A\|_\infty<\infty$, then
\begin{equation}\label{2.31}
|\widehat{A_\lambda} -\widehat{A_0}|\le C \lambda \big\{  \|\nabla^2 A\|_\infty + \|\nabla A\|^2_\infty \big\},
\end{equation}
where $C$ depends only on $d$ and  $\mu$.
\end{thm}

\begin{proof}
We first prove (\ref{2.31}).
Observe that
\begin{equation}\label{2.32}
\aligned
\widehat{A_\lambda} -\widehat{A_0}
& =\int_0^1\!\!\! \int_{\mathbb{T}^d}  A (y, s)
\nabla \big( \chi^\lambda (y, \lambda s) -\chi^0 (y) \big)\,  dyds\\
&=\int_0^1\!\!\! \int_{\mathbb{T}^d}
(A(y, s)-\overline{A}(y) )
\nabla \chi^\lambda (y, \lambda s)\, dyds\\
& \qquad
+ \int_{\mathbb{T}^d}
\overline{A} (y)
\nabla \left( \int_0^1 \chi^\lambda (y, \lambda s)\, ds
-\chi^0 (y) \right) dy\\
&=I_1 +I_2.
\endaligned
\end{equation}
Write
$ A(y, s) -\overline{A}(y) = \partial_s \widetilde{A} (y, s)$, where
$$
\widetilde{A}(y, s)=\int_0^s \big( A(y, \tau)-\overline{A} (y) \big) d\tau.
$$
Since  $\widetilde{A}(y, s)$ is 1-periodic in $(y, s)$, we may use  an integration by parts and the Cauchy inequality  to obtain
\begin{equation}\label{2.33}
|I_1|
\le C \lambda 
\left(\fint_0^\lambda \!\!\!\int_{\mathbb{T}^d}
|\partial_s \nabla \chi^\lambda|^2\, dyds\right)^{1/2}.
\end{equation}

To bound the term $I_2$ in (\ref{2.32}),
we observe that
$$
-\text{\rm div}
\left( \int_0^1 A(y, s)\nabla \chi_j^\lambda (y, \lambda s)\, ds \right)
=\text{\rm div} \left( \overline{A}( y)\nabla y_j \right)
=-\text{\rm div} \left( \overline {A} (y) \nabla \chi_j^0 (y) \right).
$$
It follows that
$$
\aligned
& -\text{\rm div} \left\{  \overline{A} (y)
\nabla \left(\int_0^1 \chi_j^\lambda (y, \lambda s)\, ds -\chi_j^0(y)  \right) \right\} \\
&
\quad =\text{\rm div}
\left\{ \int_0^1 \left( A(y, s)-\overline{A}(y) \right) \nabla \chi_j^\lambda (y, \lambda s) \, ds \right\}.
\endaligned
$$
By the energy estimates we obtain 
$$
\aligned
& \| \nabla \left(\int_0^1 \chi_j^\lambda (y, \lambda s)\, ds -\chi_j^0(y)  \right)  \|_{L^2 (\mathbb{T}^d)}\\
&\quad  \le C \| 
\left\{ \int_0^1 \left( A(y, s)-\overline{A}(y) \right) \nabla \chi_j^\lambda (y, \lambda s) \, ds \right\}\|_{L^2(\mathbb{T}^d)}\\
&\quad \le C\lambda  \left(\fint_0^\lambda \!\!\!\int_{\mathbb{T}^d}
|\partial_s \nabla \chi_j^\lambda|^2\, dyds\right)^{1/2},
\endaligned
$$
where, for the last step, we have used the integration by parts as in the estimate of $I_1$.
As a result, in view of (\ref{2.32}) and (\ref{2.33}),
we have proved that
\begin{equation}\label{2.34}
| \widehat{A_\lambda} -\widehat{A_0}|
\le  C\lambda  \left(\fint_0^\lambda \!\!\!\int_{\mathbb{T}^d}
|\partial_s \nabla \chi^\lambda|^2\, dyds\right)^{1/2}.
\end{equation}

To bound the right-hand side of (\ref{2.34}),
we differentiate in $y$ the parabolic equation  for $\chi_j^\lambda$ to obtain
\begin{equation}\label{2.35}
\partial_s \nabla \chi_j^\lambda
-\text{\rm div} \big( A_\lambda \nabla (\nabla \chi_j^\lambda) \big)
=\text{\rm div}\big( \nabla A_\lambda \cdot \nabla \chi_j^\lambda\big)
+\text{\rm div} \big( \nabla A_\lambda \cdot \nabla y_j\big).
\end{equation}
By the energy estimates,
\begin{equation}\label{2.36}
\fint_0^\lambda\!\!\! \int_{\mathbb{T}^d}
|\nabla^2 \chi_j^\lambda|^2\, dyds
\le C \|\nabla A\|_\infty^2.
\end{equation}
By differentiating (\ref{2.35}) in $y$ we have
$$
\aligned
& \partial_s \nabla^2 \chi_j^\lambda
-\text{\rm div} \big (A_\lambda \nabla (\nabla^2 \chi_j^\lambda)\big)\\
& =\text{\rm div} \big( \nabla A_\lambda \cdot \nabla^2 \chi_j^\lambda\big)
+\text{\rm div} \big( \nabla^2 A_\lambda \cdot \nabla \chi_j^\lambda\big)
+\text{\rm div} \big( \nabla A_\lambda  \cdot \nabla^2 \chi_j^\lambda)
+\text{\rm div}
\big( \nabla^2 A_\lambda \cdot \nabla y_j).
\endaligned
$$
Again, by the energy estimates,
$$
\aligned
\fint_0^\lambda\!\!\! \int_{\mathbb{T}^d}
|\nabla^3 \chi_j^\lambda|^2\, dyds
 & \le C \|\nabla A\|_\infty^2 
\fint_0^\lambda\!\!\! \int_{\mathbb{T}^d}
|\nabla^2  \chi_j^\lambda|^2\, dyds
+ C\fint_0^\lambda\!\!\! \int_{\mathbb{T}^d} |\nabla^2 A_\lambda|^2 |\nabla \chi_j^\lambda|^2\, dyds\\
& \qquad +C \fint_0^\lambda\!\!\! \int_{\mathbb{T}^d}
|\nabla^2 A_\lambda|^2\, dyds\\
&\le C \left\{ \|\nabla A\|_\infty^4
+ \|\nabla^2 A\|_\infty^2 \right\}.
\endaligned
$$
It follows by the equation (\ref{2.35}) that
$$
\fint_0^\lambda \!\!\!\int_{\mathbb{T}^d} |\partial_s \nabla\chi^\lambda|^2\, dyds
\le  C \Big\{ \|\nabla A\|_\infty^4
+ \|\nabla^2 A\|_\infty^2 \Big\},
$$
which, together with (\ref{2.34}), gives (\ref{2.31}).

Finally, to see (\ref{2.30}), we let $D$ be a smooth matrix satisfying (\ref{ellipticity}) and
(\ref{periodicity}).
As in the proof of Theorem \ref{theorem-2.1},
we have
$$
\aligned
|\widehat{A_\lambda}-\widehat{A_0}|
&\le |\widehat{A_\lambda}-\widehat{D_\lambda}|
+|\widehat{D_\lambda}-\widehat{D_0}|
+|\widehat{D_0} -\widehat{A_0}|\\
& \le C \left(\int_0^1 \!\!\!\int_{\mathbb{T}^d}
|A-D|^{p_0}\, dy ds\right)^{1/p_0}
+ C \lambda \Big\{ \|\nabla^2 D\|_\infty
+ \|\nabla D\|^2_\infty\Big\}.
\endaligned
$$
By letting $\lambda \to 0$ and by approximating $A$ in the $L^{p_0}(\mathbb{T}^{d+1})$ norm 
by a sequence of smooth matrices satisfying (\ref{ellipticity}) and (\ref{periodicity}),
we conclude that $\widehat{A_\lambda} \to \widehat{A_0}$ as $\lambda \to 0$.
\end{proof}

\begin{remark}
It  follows from the proof of Theorem \ref{theorem-2.2} that if $r\ge \e$, 
\begin{equation}\label{Q-e-6}
\aligned
& \left(\fint_{Q_r}
|\nabla \chi^\lambda (y/\e, s/\e) -
\nabla \chi^0 (y/\e, s/(\lambda \e^2)) |^2\, dyds \right)^{1/2}\\
& \qquad \qquad \qquad+
\left(\fint_{Q_r}
| \chi^\lambda (y/\e, s/\e^2) -
 \chi^0 (y/\e, s/(\lambda\e^2)) |^2\, dyds \right)^{1/2}\\
& \qquad\quad
\le C \lambda
\Big\{ \|\nabla^2 A\|_\infty
+\|\nabla A\|_\infty^2 \Big\}
\endaligned
\end{equation}
for $0<\lambda<1$, where $C$ depends only on $d$ and $\mu$.
\end{remark}



\section{Approximation}\label{section-3}

Let $A_\lambda$ be the matrix given by (\ref{A-Lambda}) and
$
\mathcal{L}_{\e, \lambda} =-\text{\rm div} \big( A_\lambda (x/\e, t/\e^2)\nabla \big).
$
Let $\mathcal{L}_{0, \lambda}
=-\text{\rm div} \big( \widehat{A_\lambda} \nabla )$, where the constant matrix 
$\widehat{A_\lambda}$ is given by (\ref{2.4}).
The goal of this section is to prove the following theorem.

\begin{thm}\label{theorem-3.1}
Suppose $A$ satisfies conditions (\ref{ellipticity}) and (\ref{periodicity}).
Let $u_{\e, \lambda}$ be a weak solution of
\begin{equation}\label{3.1}
(\partial_t +\mathcal{L}_{\e, \lambda } ) u_{\e, \lambda}  =F \quad \text{ in } Q_{2r},
\end{equation}
where $r> (1+\sqrt{\lambda}) \e$ and $F\in L^p(Q_{2r})$ for some $p>d+2$.
Then there exists a weak solution of
\begin{equation}\label{3.2}
(\partial_t +\mathcal{L}_{0, \lambda} ) u_{0, \lambda}  =F \quad \text{ in } Q_r,
\end{equation}
such that
\begin{equation}\label{3.3-1}
\left(\fint_{Q_{r}} |\nabla u_{0, \lambda} |^2 \right)^{1/2}
\le C \left(\fint_{Q_{2r}} |\nabla u_{\e, \lambda} |^2 \right)^{1/2},
\end{equation}
and
\begin{equation}\label{3.3}
\aligned
 & \left(\fint_{Q_{r/2}} |\nabla u_{\e, \lambda}  -\nabla u_{0, \lambda} 
-(\nabla \chi^\lambda)^\e \nabla u_{0, \lambda}  |^2  \right)^{1/2}\\
& \le C \left(\frac{(1+\sqrt{\lambda})\e }{r} \right)^\sigma 
\left\{
\left(\fint_{Q_{2r}} |\nabla u_{\e, \lambda} |^2  \right)^{1/2}
+   r \left(\fint_{Q_{2r}} |F|^p\right)^{1/p} \right\},
\endaligned
\end{equation}
where $\sigma\in (0, 1)$ and $C>0$ depend only on $d$, $\mu$ and $p$.
\end{thm}

We begin by introducing the dual correctors $\phi^\lambda$  for the operator $\partial_t +\mathcal{L}_{\e, \lambda}$.
Let
\begin{equation}\label{B-3}
B_\lambda =A_\lambda +A_\lambda \nabla \chi^\lambda -\widehat{A_\lambda},
\end{equation}
where the corrector  $\chi^\lambda$  is  given by (\ref{cell-lambda}).
Note that $B_\lambda$ is $(1, \lambda)$-periodic in $(y, s)$.

\begin{lemma}\label{lemma-3.1}
Let $B_\lambda =( b_{ij}^\lambda)$ be given by (\ref{B-3}).
Then there exist $(1, \lambda)$-periodic functions $\phi^\lambda_{kij}$ and $\phi^\lambda_{k(d+1)j}$, with $1\le i, j, k\le d$,  in $H^1_{loc} (\mathbb{R}^{d+1})$ such that
\begin{equation}\label{3.11}
\left\{
\aligned
b_{ij}^\lambda  & = \frac{\partial }{\partial y_k} \phi^\lambda_{kij} -\partial_s \phi^\lambda_{i (d+1)j},\\
-\chi_j^\lambda &=\frac{\partial}{\partial y_k} \phi^\lambda_{k (d+1)j}.
\endaligned
\right.
\end{equation}
Moreover, $\phi_{kij}^\lambda=-\phi_{ikj}^\lambda$ and 
\begin{align}
\fint_0^\lambda\!\!\! \int_{\mathbb{T}^d}
\big( |\phi_{kij}^\lambda|^2 +|\nabla \phi^\lambda_{k (d+1) j}|^2 \big)  & \le C,\label{3.12}\\
\fint_0^\lambda\!\!\! \int_{\mathbb{T}^d}
|\phi^\lambda_{k (d+1) j}|^2   & \le C (1+\lambda)^2, \label{3.12-1}
\end{align}
where $C$ depends only on $d$ and $\mu$.
\end{lemma}

\begin{proof}
The lemma was proved in \cite{Geng-Shen-2017} for the case $\lambda=1$.
The case $\lambda\neq 1$ is similar.
However, one  needs to be careful with the dependence of the constants $C$ on the parameter $\lambda$.

Let $\Delta_{d+1}$ denote the Laplacian operator in $\mathbb{R}^{d+1}$.
By the definition of $\widehat{A_\lambda}$,
\begin{equation}\label{3.13}
\fint_0^\lambda\!\!\! \int_{\mathbb{T}^d} B_\lambda (y, s)\, dyds =0.
\end{equation}
It follows that there exist $(1,\lambda)$-periodic functions $f_{ij}^\lambda\in H^2_{loc}(\mathbb{R}^{d+1})$  such that
$
\Delta_{d+1} f_{ij}^\lambda =b_{ij}^\lambda  \text{ in } \mathbb{R}^{d+1}
$
for $1\le i, j\le d$.
Similarly, there exist $(1, \lambda)$-periodic
functions $f^\lambda_{(d+1)j}\in H^2_{loc}(\mathbb{R}^{d+1})$ such that 
$
\Delta_{d+1} f^\lambda_{(d+1)j } =-\chi_j^\lambda  \text{ in } \mathbb{R}^{d+1}
$
for $1\le j \le d$.
By the definition of $\chi^\lambda_j$,  we have
\begin{equation}\label{3.14}
\frac{\partial}{\partial y_i} b_{ij}^\lambda =\partial_s \chi_j^\lambda \quad \text{ in } \mathbb{R}^{d+1},
\end{equation}
which leads to
$$
\Delta_{d+1} \left( \frac{\partial f_{ij}^\lambda}{\partial y_i} +\partial_s f_{(d+1)j}^\lambda\right) =0
\quad \text{ in } \mathbb{R}^{d+1}.
$$
By the periodicity and Liouville Theorem  we may conclude that
\begin{equation}\label{3.15}
 \frac{\partial f_{ij}^\lambda}{\partial y_i} +\partial_s f_{(d+1)j}^\lambda
 \quad \text{ is constant in } \mathbb{R}^{d+1} \text{ for } 1\le j\le d.
 \end{equation}
This allows us to write
$$
b_{ij}^\lambda
=\frac{\partial}{\partial y_k}
\left\{ \frac{\partial f_{ij}^\lambda}{\partial y_k} 
-\frac{\partial f^\lambda_{kj}}{\partial y_i} \right\}
+\partial_s
\left\{ \partial_s f_{ij}^\lambda - \frac{\partial f^\lambda_{(d+1) j}}{\partial y_i} \right\},
$$
and
$$
-\chi_j^\lambda =\frac{\partial}{\partial y_k} 
\left\{ \frac{\partial f_{(d+1)j}^\lambda}{\partial y_k} -\partial_s f^\lambda_{kj} \right\}.
$$
We now define $\phi_{kij}^\lambda$ and $\phi_{k (d+1)j}^\lambda$ by
\begin{equation}\label{3.16}
\left\{
\aligned
\phi_{kij}^\lambda & =\frac{\partial f_{ij}^\lambda}{\partial y_k} -\frac{\partial f_{kj}^\lambda}{\partial y_i},\\
\phi_{k(d+1)j}^\lambda & =\frac{\partial f^\lambda_{(d+1)j}}{\partial y_k} -\partial_s f_{kj}^\lambda
\endaligned
\right.
\end{equation}
for $1\le i, j, k\le d$. This gives (\ref{3.11}).
It is easy to see that $\phi_{kij}^\lambda =-\phi_{ikj}^\lambda$.

Finally, to prove estimates (\ref{3.12}) and (\ref{3.12-1}),
we use the Fourier series to write
$$
b_{ij}^\lambda (y, s)
=\sum_{\substack{ n\in \mathbb{Z}^d, m\in \mathbb{Z}\\ (n, m)\neq (0, 0) }}
a_{n, m} e^{-2\pi i n\cdot y -2\pi i m s\lambda^{-1}}.
$$
Then
$$
f_{ij}^\lambda (y, s)=-\frac{1}{4\pi^2}
\sum_{\substack{ n\in \mathbb{Z}^d, m\in \mathbb{Z}\\ (n, m)\neq (0, 0) }}
\frac{ a_{n, m}}{ |n|^2 +|m|^2 \lambda^{-2}}
 e^{-2\pi i n\cdot y -2\pi i m s\lambda^{-1}}.
$$
It follows by Parseval's Theorem  that
\begin{equation}\label{3.17}
\aligned
 \fint_0^\lambda\! \!\!\!\int_{\mathbb{T}^d} & 
\Big(  |\nabla f_{ij}^\lambda|^2 
+ |\nabla^2 f_{ij}^\lambda|^2
+ |\partial_s^2 f_{ij}^\lambda|^2
+|\nabla \partial_s f_{ij}^{\lambda}|^2 \Big)  \\
 & \le C \sum_{n, m}  |a_{n, m}|^2
  = C \fint_0^\lambda \!\!\!\!\int_{\mathbb{T}^d}
 |b_{ij}^\lambda|^2
   \le C,
 \endaligned
 \end{equation}
 where $C$ depends only on $d$ and $\mu$.
 Also note that
 \begin{equation}\label{3.18}
 \fint_0^\lambda \!\!\!\!\int_{\mathbb{T}^d}
 |\partial_s f_{ij}^\lambda|^2
 \le C \lambda^2,
 \end{equation}
 where $C$ depends only on $d$ and $\mu$.
 Similarly,  using the estimate (\ref{2.3}), we obtain 
\begin{equation}\label{3.19}
\fint_0^\lambda \!\!\!\int_{\mathbb{T}^d} 
\Big(  |\nabla f_{(d+1)j}^\lambda|^2 
+ |\nabla^2 f_{(d+1)j}^\lambda|^2
+ |\partial_s^2 f_{(d+1)j}^\lambda|^2
+|\nabla \partial_s f_{(d+1)j}^{\lambda}|^2 \Big)  
\le C.
\end{equation}
The desired estimates (\ref{3.12}) and (\ref{3.12-1})
follow readily from (\ref{3.16}), (\ref{3.17}), (\ref{3.18}) and (\ref{3.19}).
\end{proof}

Let $\varphi=\varphi (y, s)=\theta_1 (y)  \theta_2 (s)$,
where $\theta_1 \in C_0^\infty (B(0, 1))$, $\theta_2\in C^\infty_0 (-1, 0)$, $\theta_1, \theta_2\ge 0$, and
$\int_{\mathbb{R}^{d}} \theta_1 (y) \, dy =\int_{\mathbb{R}} \theta_2 (s)\, ds=1$.
Define
\begin{equation}\label{S}
S_\delta (f) (x, t) =\int_{\mathbb{R}^{d+1}} f(x-y, t-s) \varphi_\delta (y, s)\, dyds,
\end{equation}
where $\delta>0$ and $\varphi_\delta (y, s)=\delta^{-d-2} \varphi (y/\delta, s/\delta^2)$.

\begin{lemma}\label{lemma-S}
Let $g\in L^2_{loc}(\mathbb{R}^{d+1})$ and $f\in L^2(\mathbb{R}^{d+1})$.
Then
\begin{align}
\| g S_\delta (f) \|_{L^2(\mathbb{R}^{d+1})}
\le C \sup_{(y, s)\in \mathbb{R}^{d+1}}
\left(\fint_{Q_{\delta} (y, s)} |g|^2 \right)^{1/2}
\| f\|_{L^2(\mathbb{R}^{d+1})},\label{S-1}\\
\| g \nabla S_\delta (f) \|_{L^2(\mathbb{R}^{d+1})} 
\le C \delta^{-1} \sup_{(y, s)\in \mathbb{R}^{d+1}}
\left(\fint_{Q_{\delta} (y, s)} |g|^2\right)^{1/2}
\| f\|_{L^2(\mathbb{R}^{d+1})}\label{S-2},
\end{align}
where $C$ depends only on $d$.
\end{lemma}

\begin{proof}
By H\"older's inequality,
$$
|S_\delta (f) (x, t)|^2
\le \int_{\mathbb{R}^{d+1}} | f(y, s)|^2 \varphi_\delta (x-y, t-s)\, dyds.
$$
It follows by Fubini's Theorem that
$$
\aligned
\int_{\mathbb{R}^{d+1}} |g|^2 |S_\delta (f)|^2\, dx dt
&\le \int_{\mathbb{R}^{d+1}}
|f(y, s)|^2 \left( \int_{\mathbb{R}^{d+1}}
|g(x, t)|^2 \varphi_\delta (x-y, t-s)\, dx dt \right) dy ds\\
& \le C \sup_{(y, s)\in \mathbb{R}^{d+1}}
\left( \fint_{Q_{\delta} (y, s)} |g|^2\right)  
\| f\|^2_{L^2(\mathbb{R}^{d+1})},
\endaligned
$$
where $C$ depends only on $d$.
This gives (\ref{S-1}).
The estimate (\ref{S-2}) follows in a similar manner.
\end{proof}

\begin{lemma}\label{lemma-3.2}
Let $S_\delta$ be define by (\ref{S}).
Then
\begin{equation}\label{S-3}
\aligned
\| g \nabla f - S_\delta (g \nabla f)\|_{L^2(\mathbb{R}^{d+1})}
 & \le C \delta
\Big\{
\| \nabla (g \nabla f) \|_{L^2(\mathbb{R}^{d+1})}
+ \|g \partial_t  f\|_{L^2(\mathbb{R}^{d+1})}\\
&\qquad\quad
+ \delta \| (\partial_t g) (\nabla f) \|_{L^2(\mathbb{R}^{d+1})}
+ \delta \|(\nabla g ) \partial_t  f\|_{L^2(\mathbb{R}^{d+1})}
\Big\},
\endaligned
\end{equation}
where $C$ depends only on $d$.
\end{lemma}

\begin{proof}
Write $S_\delta =S_\delta^1 S_\delta^2$, where
\begin{equation}\label{S-1-2}
\left\{
\aligned
 & S_\delta^1 (f) (x, t)=\int_{\mathbb{R}^d} f(x-y, t) \delta^{-d} \theta_1 (y/\delta)\, dy,\\
& S_\delta^2 (f) (x, t)=\int_{\mathbb{R}}  f(x, t-s) \delta^{-2} \theta_2 (s/\delta^2)\, ds.
\endaligned
\right.
\end{equation}
By using the Plancherel Theorem, it is easy to see that 
$$
\left\{
\aligned
\| f -S_\delta^1 (f)\|_{L^2(\mathbb{R}^{d+1})}
& \le C \delta \| \nabla f\|_{L^2(\mathbb{R}^{d+1})},\\
\| f -S_\delta^2 (f)\|_{L^2(\mathbb{R}^{d+1})}
& \le C \delta^2 \| \partial_t f\|_{L^2(\mathbb{R}^{d+1})},
\endaligned
\right.
$$
where $C$ depends only on $d$.
It follows that
$$
\aligned
\| g \nabla f -S_\delta (g \nabla f) \|_{L^2(\mathbb{R}^{d+1})}
& \le \| g \nabla f -S_\delta^1 ( g\nabla f) \|_{L^2(\mathbb{R}^{d+1})}
+ \| S_\delta^1 ( g \nabla f) - S_\delta ( g\nabla f) \|_{L^2(\mathbb{R}^{d+1})}\\
&\le C \delta \| \nabla (g \nabla f ) \|_{L^2(\mathbb{R}^{d+1})}
+ C \delta^2 \|\partial_t S_\delta^1 (g\nabla f) \|_{L^2(\mathbb{R}^{d+1})}.
\endaligned
$$
To bound the last term in the inequalities above, we note that
$$
\partial_t (g \nabla f)
= (\partial_t g)  \nabla  f
+\nabla (g \partial_t f)
- (\nabla g) \partial_t f.
$$
Using the estimates 
$$
\| S_\delta^1 (h)\|_{L^2(\mathbb{R}^{d+1})}
\le  \| h\|_{L^2(\mathbb{R}^{d+1})} \quad \text{ and } \quad
\|\nabla S_\delta^1 (h)\|_{L^2(\mathbb{R}^{d+1})}
\le C \delta^{-1} \| h\|_{L^2(\mathbb{R}^{d+1})},
$$
we obtain 
$$
\| \partial_t S_\delta^1 (g \nabla f)\|_{L^2(\mathbb{R}^{d+1})}
\le \| (\partial_t g ) \nabla f\|_{L^2(\mathbb{R}^{d+1})}
+ C \delta^{-1} \| g \partial_t f\|_{L^2(\mathbb{R}^{d+1})}
+ \| (\nabla g) \partial_t  f\|_{L^2(\mathbb{R}^{d+1})}.
$$
This completes the proof.
\end{proof}

Let
\begin{equation}\label{w-3}
w_\e
=u_{\e, \lambda}  -u_{0, \lambda}  - \e ( \chi_j^\lambda)^\e K_\e \left(\frac{\partial u_{0, \lambda} }{\partial x_j}\right)
+\e^2 \left( \phi_{i (d+1) j}^\lambda\right)^\e \frac{\partial }{\partial x_i} 
K_\e \left(\frac{\partial u_{0, \lambda} }{\partial x_j} \right),
\end{equation}
where  
$$
(\chi_j^\lambda)^\e =\chi_j^\lambda (x/\e, t/\e^2),\ \ 
\ (\phi_{i (d+1)j}^\lambda)^\e = \phi_{i (d+1) j}^\lambda (x/\e, t/\e^2),
$$
 and $K_\e$ is a linear operator to be specified later .

\begin{lemma}\label{lemma-3.3}
Suppose that
$$
(\partial_t +\mathcal{L}_{\e, \lambda })u_{\e, \lambda}
=
(\partial_t +\mathcal{L}_{0, \lambda} ) u_{0, \lambda}  \quad \text{\ in } \Omega \times (T_0, T_1).
$$
Let $w_\e$ be defined by (\ref{w-3}).
Then
\begin{equation}\label{w-3.1}
\aligned
   (\partial_t +\mathcal{L}_{\e, \lambda}) w_\e
& = -\text{\rm div} \left( \big( \widehat{A_\lambda} -A_\lambda (x/\e, t/\e^2)\big)
\big(\nabla u_{0, \lambda}  -K_\e (\nabla u_{0, \lambda}) \big)\right)\\
&\quad +\e\,  \text{\rm div}
\Big( A_\lambda (x/\e, t/\e^2)\chi^\lambda (x/\e, t/\e^2) \nabla K_\e (\nabla u_{0, \lambda} ) \Big)\\
&\quad +\e\,  \frac{\partial }{\partial x_k}
\left\{ \phi_{kij}^\lambda (x/\e, t/\e^2) \frac{\partial }{\partial x_i} K_\e \left(\frac{\partial u_{0, \lambda} }{\partial x_j}\right) \right\}\\
&\quad + \e^2 \frac{\partial}{\partial x_k}
\left\{ \phi^\lambda_{k (d+1) j} (x/\e, t/\e^2) \partial_t K_\e \left(\frac{\partial u_{0, \lambda} }{\partial x_j}\right) \right\}\\
&\quad -\e \, \frac{\partial}{\partial x_i}
\left\{ a_{ij}^\lambda (x/\e, t/\e^2) \left( \frac{\partial}{\partial x_j} \phi_{\ell (d+1) k}^\lambda \right) (x/\e, t/\e^2)
\frac{\partial}{\partial x_\ell}
K_\e \left( \frac{\partial u_{0, \lambda} }{\partial x_k} \right) \right\}\\
&\quad -\e^2\frac{\partial}{\partial x_i}
\left\{ a_{ij}^\lambda (x/\e, t/\e^2) \phi^\lambda_{\ell (d+1) k }(x/\e, t/\e^2)
\frac{\partial^2}{\partial x_j \partial x_\ell}
K_\e \left( \frac{\partial u_{0, \lambda} }{\partial x_k} \right) \right\},
\endaligned
\end{equation}
where $A_\lambda = \big( a_{ij}^\lambda \big)$.
\end{lemma}

\begin{proof}
This is proved by a direct computation.
See \cite[Theorem 2.2]{Geng-Shen-2017} for the case $\lambda=1$.
\end{proof}

\begin{lemma}\label{lemma-3.4}
Let $Q_r=B(0, r) \times (-r^2, 0)$.
Suppose $u_{\e, \lambda}$ is a weak solution of
$( \partial_t +\mathcal{L}_{\e, \lambda} ) u_{\e, \lambda} =F$ in $Q_2$
for some $F\in L^2(Q_2)$.
Then there exists a weak solution  of
$(\partial_t +\mathcal{L}_{0, \lambda} ) u_{0, \lambda} =F$ in $Q_1$
such that
\begin{equation}\label{3.40-1}
\left(\fint_{Q_1} |\nabla u_{0, \lambda}|^2 \right)^{1/2}
\le C \left(\fint_{Q_2} |\nabla u_{\e, \lambda} |^2 \right)^{1/2},
\end{equation}
and for $\delta = (1+\sqrt{\lambda} ) \e$,
\begin{equation}\label{3.40}
\aligned
& \left(\fint_{Q_1}\big |\nabla  \Big( u_{\e, \lambda} - u_{0, \lambda} -\e \chi^\lambda (x/\e, t/\e^2) K_\e (\nabla u_{0, \lambda} )\Big)  \big|^2\, dx dt  \right)^{1/2}\\
& \le C \delta^\sigma \left\{
\left(\fint_{Q_2} |\nabla u_{\e, \lambda}|^2 \right)^{1/2}
+   \left(\fint_{Q_2} |F|^2 \right)^{1/2} 
\right\},
\endaligned
\end{equation}
where 
$\sigma \in (0, 1)$ and $C>0$ depend only on $d$ and $\mu$.
The operator $K_\e$ is defined by (\ref{K}).
\end{lemma}

\begin{proof}
We start out by defining $u_{0, \lambda} $ to be the  weak solution of the initial-Dirichlet problem:
\begin{equation}\label{3.41}
\left\{
\aligned
(\partial_t + \mathcal{L}_{0, \lambda} ) u_{0, \lambda}  &=F & \quad & \text{ in } Q_{1},\\
u_{0, \lambda} & =u_{\e, \lambda} & \quad & \text{ on } \partial_p Q_{1},
\endaligned
\right.
\end{equation}
where $\partial_p Q_{1}$ denotes the parabolic boundary of the cylinder $Q_{1}$.
Note that
$$
(\partial_t +\mathcal{L}_{0, \lambda} ) (u_{0, \lambda} -u_{\e, \lambda} )
=(\mathcal{L}_{\e, \lambda} -\mathcal{L}_{0, \lambda}) u_{\e, \lambda}
$$
in $Q_{1}$ and $u_{\e, \lambda}  -u_{0, \lambda} =0$ on $\partial_p Q_{1}$.
It follows from the standard regularity estimates for parabolic operators with constant coefficients that
$$
\fint_{Q_{1} }|\nabla (u_{\e, \lambda} -u_{0, \lambda})|^q 
\le C \fint_{Q_{1}} |\nabla u_{\e, \lambda} |^q 
$$
for any $2\le q<\infty$, where $C$ depends only on $d$, $\mu$ and $q$.
This gives 
$$
\fint_{Q_{1} }|\nabla u_{0, \lambda} |^q 
\le C \fint_{Q_{1}} |\nabla u_{\e, \lambda} |^q 
$$
for any $2<q<\infty$.
By the Meyers-type estimates for parabolic systems \cite[Appendix]{Armstrong-2018},
there exist some $q>2$ and $C >0$, depending on $d$ and $\mu$, such that
$$
\left(\fint_{Q_{1}} |\nabla u_{\e, \lambda} |^q \right)^{1/q}
\le C \left\{ \left(\fint_{Q_2} |\nabla u_{\e, \lambda}|^2\right)^{1/2}
+ \left(\fint_{Q_2} |F|^2 \right)^{1/2} \right\}.
$$
As a result, we obtain 
\begin{equation}\label{3.42}
\left(\fint_{Q_{1}} |\nabla u_{0, \lambda}|^q  \right)^{1/q}
\le  C \left\{ \left(\fint_{Q_2} |\nabla u_{\e, \lambda} |^2\right)^{1/2}
+ \left(\fint_{Q_2} |F|^2 \right)^{1/2} \right\}
\end{equation}
for some $q>2$ and $C>0$, depending only on $d$ and $\mu$.

To prove (\ref{3.40}), we let $\delta=(1+\sqrt{\lambda})\e $.
We may assume 
$\delta\le 1/8$;
for otherwise the estimate is trivial. Choose $\eta_\delta\in C_0^\infty(\mathbb{R}^{d+1})$ such that
$0\le \eta_\delta\le 1$, $\ |\nabla \eta_\delta|\le C/\delta$, 
$ |\partial_t \eta_\delta| +|\nabla^2 \eta_\delta|\le C /\delta^2$,
$$
\eta_\delta =1 \quad \text{ in } Q_{1-3\delta} \quad \text{ and } \quad 
\eta_\delta=0  \quad \text { in } Q_1 \setminus Q_{1-2\delta}.
$$
Let $w_\e$ be defined by (\ref{w-3}), where the operator $K_\e$ is given by
\begin{equation}\label{K}
K_\e (f) = S_\delta ( \eta_\delta f)
\end{equation}
with $S_\delta$ defined  in (\ref{S}).
Note that $w_\e=0$ in $\partial_p Q_1$.
It follows from Lemma \ref{lemma-3.3} and energy estimates that
\begin{equation}\label{3.43}
\aligned
\int_{Q_1}|\nabla w_\e|^2
& \le C \int_{Q_1} |\nabla u_{0, \lambda} -K_\e (\nabla u_{0, \lambda}) |^2
+ C \e^2  \int_{Q_1} |( \chi^\lambda)^\e \nabla K_\e (\nabla u_{0, \lambda})|^2\\
&\qquad + C \e^2 \int_{Q_1}
\sum_{k,i,j} |(\phi_{kij}^\lambda)^\e|^2 |\nabla K_\e (\nabla u_{0, \lambda})|^2\\
&\qquad + C\e^4 \int_{Q_1}\sum_{k, j}  |(\phi_{k (d+1)j}^\lambda)^\e|^2 |\partial_t K_\e (\nabla u_{0, \lambda} )|^2\\
& \qquad + C \e^2 \int_{Q_1}
\sum_{\ell, k} | (\nabla \phi_{\ell (d+1) k} ^\lambda)^\e|^2 |\nabla K_\e (\nabla u_{0, \lambda} )|^2\\
& \qquad + C\e^4 \int_{Q_1}
\sum_{\ell, k} |(\phi_{\ell (d+1) k}^\lambda)^\e|^2 
|\nabla^2 K_\e (\nabla u_{0, \lambda} )|^2\\
&=I_1 +I_2 +I_3 +I_4 +I_5+I_6.
\endaligned
\end{equation}
To bound $I_1$, we use Lemma \ref{lemma-3.2}.
This gives
$$
\aligned
I_1 & \le 
2\int_{Q_1} |\nabla u_{0, \lambda} -\eta_\delta  (\nabla u_{0, \lambda})|^2
+ 2\int_{Q_1} |\eta_\delta  (\nabla u_{0, \lambda}) -S_\delta (\eta _\delta (\nabla u_{0, \lambda}))|^2\\
&\le C \int_{Q_1 \setminus  Q_{1-3\delta}} |\nabla u_{0, \lambda}|^2
+ C \delta^2  \int_{Q_{1-2\delta}}\big(  |\nabla^2 u_{0, \lambda} |^2 +|\partial_t u_{0, \lambda}|^2 \big).
\endaligned
$$
By the standard regularity estimates for parabolic systems with constant coefficients,
$$
\int_{Q_{1-2\delta}}
\big( |\nabla^2 u_{0, \lambda}|^2 + |\partial_t u_{0, \lambda}|^2 \big) 
\le C \left\{ \int_{Q_{1-\delta}}\frac{|\nabla u_{0, \lambda}( y, s) |^2\, dy ds  }{| \text{dist} _p ((y,s), \partial_p Q_1)|^2 }
+\int_{Q_1}  |F|^2  \right\},
$$
where $ \text{\rm dist}_p ((y, s), \partial_p Q_1)$ denotes the parabolic distance from $(y, s)$ to $\partial_p Q_1$.
It follows that 
\begin{equation}\label{3.44}
\aligned
I_1 & \le C \int_{Q_1 \setminus  Q_{1-3\delta}} |\nabla u_{0, \lambda}|^2
+C\delta^2  \left\{ \int_{Q_{1-\delta}}\frac{|\nabla u_{0, \lambda}( y, s) |^2\, dy ds  }{| \text{dist} _p ((y,s), \partial_p Q_1)|^2 }
+\int_{Q_1}  |F|^2  \right\}\\
& \le C \delta^{1-\frac{2}{q}}
\left(\fint_{Q_1} |\nabla u_{0, \lambda}|^q \right)^{2/q}
+ C \delta^2 \fint_{Q_1} |F|^2,
\endaligned
\end{equation}
where $q>2$ and we have used H\"older's inequality for the last step.

To bound $I_2$, $I_3$ and $I_5$, we use Lemma \ref{lemma-S} as well as  estimates (\ref{2.3}) and (\ref{3.12}),
Note that $(\chi^\lambda)^\e$, $(\phi_{kij}^\lambda)^\e$ and
$(\nabla \phi_{\ell (d+1) k}^\lambda)^\e$
are $\e$-periodic in $x$ and $\e^2 \lambda$-periodic in $t$.
Since $\delta\ge \e$ and $\delta^2 \ge \e^2 \lambda$, we obtain 
$$
\aligned
 & \fint_{Q_\delta (x, t)}
\Big(  | (\chi^\lambda)^\e|^2
+ |(\phi_{kij}^\lambda)^\e|^2
+ |( \nabla \phi^\lambda_{\ell (d+1) k})^\e  |^2 \Big)   \\
& \quad \le
C \fint_0^\lambda \!\!\!  \int_{\mathbb{T}^d} 
\Big(  | \chi^\lambda |^2
+ |\phi^\lambda _{kij} |^2
+ |\nabla \phi^\lambda_{\ell (d+1) k} |^2 \Big)  \\
&\quad \le C
\endaligned
$$
for any $(x, t)\in \mathbb{R}^{d+1}$.
It follows that
\begin{equation}\label{3.45}
\aligned
I_2+I_3 +I_5
& \le C\e^2  \int_{Q_1} |\nabla (\eta_\delta (\nabla u_{0, \lambda} )) |^2\\
& 
 \le C \delta^{1-\frac{2}{q}}
\left(\fint_{Q_1} |\nabla u_{0, \lambda}|^q \right)^{2/q}
+ C \delta^2 \fint_{Q_1} |F|^2.
\endaligned
\end{equation}
To bound $I_6$, we use the  inequality (\ref{S-2}) as well as the estimate (\ref{3.12-1}).
This leads to
\begin{equation}\label{3.46}
\aligned
I_6   & \le C \e^4 (1+\lambda)^2 \delta^{-2} \int_{Q_1} |\nabla (\eta_\delta \nabla u_{0, \lambda})|^2 \\ 
&
\le C \delta^{1-\frac{2}{q}}
\left(\fint_{Q_1} |\nabla u_{0, \lambda} |^q \right)^{2/q}
+ C \delta^2 \fint_{Q_1} |F|^2.
\endaligned
\end{equation}

Finally, to handle $I_4$, we use the observation
\begin{equation}\label{3.47}
\aligned
\partial_t K_\e (\nabla u_{0, \lambda})
&= \partial_t S_\delta (\eta_\delta \nabla u_{0, \lambda})\\
&= S_\delta ( (\partial_t \eta_\delta) \nabla u_{0, \lambda})
+ S_\delta (\nabla (\eta_\delta \partial_t u_{0, \lambda}))
+S_\delta ( (\nabla \eta_\delta) \partial_t u_{0, \lambda}).
\endaligned
\end{equation}
As in the case of $I_6$, we obtain
\begin{equation}\label{3.48}
\aligned
I_4
&\le C \e^4 (1+\lambda)^2
 \int_{Q_1} 
 \Big\{
 |(\partial_t \eta_\delta ) \nabla u_{0, \lambda}|^2
+ \delta^{-2} | \eta_\delta \partial_t u_{0, \lambda}|^2
+ |(\nabla \eta_\delta) \partial_t u_{0, \lambda}|^2  \Big\}  \\
&\le 
C \delta^{1-\frac{2}{q} }
\left(\fint_{Q_1} |\nabla u_{0, \lambda} |^q \right)^{2/q}
+ C \delta^2 \fint_{Q_1} |F|^2.
\endaligned
\end{equation}
Let  $\sigma =\frac12-\frac{1}{q}>0$.
In view of (\ref{3.44}), (\ref{3.45}), (\ref{3.46}) and (\ref{3.47}), we have proved that
\begin{equation}\label{3.49}
\aligned
\fint_{Q_1} |\nabla w_\e|^2
 & \le C \delta^{2\sigma }
\left(\fint_{Q_1} |\nabla u_{0, \lambda}|^q \right)^{2/q}
+ C \delta^2 \fint_{Q_1} |F|^2\\
&\le C\delta^{2\sigma}
\left\{ \fint_{Q_2} |\nabla u_{\e, \lambda} |^2 
+ \fint_{Q_2} |F|^2\ \right\},
\endaligned
\end{equation}
where.we have used (\ref{3.42}) for the last step.
To finish the proof, we let $H_\e$ be the last term in (\ref{w-3}).
It is easy to see that
$$
\int_{Q_1} |\nabla H_\e|^2
\le I_5 + I_6.
$$
This, together with (\ref{3.49}), gives the estimate (\ref{3.40}).
\end{proof}

We are now ready to give the proof of Theorem \ref{theorem-3.1}.

\begin{proof}[\bf Proof of Theorem \ref{theorem-3.1}]
By translation and dilation we may assume that $r=1$ and $Q_2=B(0, 2) \times (-4, 0)$.
We may also assume that $\delta=(1+\sqrt{\lambda}) \e \le 1/8$.
This reduces the problem to the case considered in Lemma \ref{lemma-3.4}.
Observe that  $K_\e (\nabla u_{0, \lambda}) =S_\delta  (\nabla u_{0, \lambda})$ on $Q_{1/2}$.
Thus, in view of Lemma \ref{lemma-3.4},
it suffices to show that
\begin{equation} \label{I-11}
\left(\fint_{Q_{1/2}} 
\big| \nabla\Big\{
\e (\chi^\lambda)^\e \ S_\delta  (\nabla u_{0, \lambda}) \Big\}  -(\nabla \chi^\lambda)^\e \nabla u_{0, \lambda}  \big|^2\right)^{1/2}
\end{equation}
is bounded by the right-hand side of (\ref{3.40}).
Furthermore, since $(\partial_t +\mathcal{L}_{0, \lambda} )u_{0, \lambda} =F $ in $Q_1$, we have
$$
\|\nabla^2 u_{0, \lambda}\|_{L^2(Q_{3/4})}
\le C \left\{ \left(\fint_{Q_1} |\nabla u_{0, \lambda}|^2 \right)^{1/2}
+\left(\fint_{Q_1} |F|^2 \right)^{1/2} \right\}.
$$
Also, recall  that
\begin{equation}\label{J-0}
\|(\chi^\lambda)^\e\|_{L^2(Q_1)} +\| (\nabla \chi^\lambda)^\e\|_{L^2(Q_1)}\le C.
\end{equation}
As a result, it is enough to show that 
\begin{equation}\label{J-1}
\left(\fint_{Q_{1/2}}
\big|  (\nabla \chi^\lambda)^\e \big( S_\delta (\nabla u_{0, \lambda}) - \nabla u_{0, \lambda}\big) |^2 \right)^{1/2}
\end{equation}
 is bounded by the right-hand side of (\ref{3.40}).
This, however, follows   from (\ref{J-0}) and  the estimate
\begin{equation}\label{J-2}
\|S_\delta (\nabla u_{0, \lambda}) -\nabla u_{0, \lambda}\|_{L^\infty(Q_{1/2})}
\le C \delta^\sigma
\left\{ 
\left(\fint_{Q_1} |\nabla u_{0, \lambda}|^2 \right)^{1/2}
+ \left(\fint_{Q_1} |F|^p \right)^{1/p} \right\},
\end{equation}
where $p>d+2$ and $\sigma =1-\frac{d+2}{p}$.

Finally, we point out that (\ref{J-2}) follows readily from  the $C^{1+ \sigma}$ estimates for $\partial_t +\mathcal{L}_{0, \lambda}$,
\begin{equation}
\aligned
&|\nabla u_{0, \lambda}(x,t)-\nabla u_{0, \lambda} (y, s)|\\
 & \le C \Big(
 |x-y|  + |t-s|^{1/2}\Big)^\sigma
 \left\{ \left( \fint_{Q_1} |\nabla u_{0, \lambda} |^2 \right)^{1/2}
+ \left(\fint_{Q_1} |F|^p\right)^{1/2} \right\}
\endaligned
\end{equation}
for any $(x, t), (y, s)\in Q_{1/2}$.
This completes the proof.
\end{proof}



\section{Large-scale  Lipschitz and $C^{1, \alpha}$  estimates}\label{section-4}

In this section we establish the large-scale Lipschitz and $C^{1, \alpha}$ estimates for 
$\partial_t +\mathcal{L}_{\e, \lambda}$.
As a consequence, we obtain the same estimates for the parabolic operator $\partial_t+ \mathcal{L}_\e$ in (\ref{operator-0}).
Let 
\begin{equation}\label{P-1}
\aligned
P^\lambda_{1, \e} =
\Big\{ P=P(x, t): \ & P(x, t)=\beta + e_j (x_j +\e \chi_j^\lambda (x/\e, t/\e^2) )\\
&  \text{ for some } \beta \in \mathbb{R} \text{ and } ( e_1,  e_2,\dots, e_d) \in \mathbb{R}^{d} \Big \}.
\endaligned
 \end{equation}
 Note  that $(\partial_t +\mathcal{L}_{\e, \lambda}) P=0$ in $\mathbb{R}^{d+1}$ for any
 $P\in P_{1, \e}^\lambda$.

\begin{thm}[$C^{1, \alpha}$ estimate]
\label{theorem-4.1}
Suppose $A$ satisfies conditions (\ref{ellipticity}) and (\ref{periodicity}).
Let $u_{\e, \lambda}$ be a weak solution of $(\partial_t +\mathcal{L}_{\e, \lambda}) u_{\e, \lambda} =F$
in $Q_R$, where $R> (1+\sqrt{\lambda}) \e$ and $F\in L^p(Q_R)$ for some $p>d+2$.
Then, for any $(1+\sqrt{\lambda})\e \le r<  R$ and $0<\alpha< 1-\frac{d+2}{p}$,
\begin{equation}\label{4.1-0}
\inf_{P\in P_{1, \e}^\lambda}
\left(\fint_{Q_r}
|\nabla (u_{\e, \lambda} -P )|^2 \right)^{1/2}
\le C \left(\frac{r}{R}\right)^\alpha
\left\{ \left(\fint_{Q_R} |\nabla u_{\e, \lambda}|^2 \right)^{1/2}
+ R \left(\fint_{Q_R} |F|^p \right)^{1/p} \right\},
\end{equation}
where $C>0$  depend only on $d$, $\mu$, $p$ and $\alpha$.
\end{thm}

\begin{proof}
The proof relies on the approximation results in Theorem \ref{theorem-3.1} and uses 
classical regularity estimates for parabolic systems with constant coefficients.
By translation and dilation we may assume that $R=2$ and $Q_2 =B(0, 2)\times (-4, 0)$.
Let 
$$
(1+\sqrt{\lambda} )\e < \theta r< r< 1,
$$
where $\theta\in (0, 1/4)$ is to be chosen later.
Let $u_{0, \lambda}$ be the weak solution of $(\partial_t +\mathcal{L}_{0, \lambda}) u_{0, \lambda} =F$ in $Q_r$, given by Theorem \ref{theorem-3.1}.
By the classical $C^{1+\alpha}$ estimates for parabolic systems with constant coefficients,
$$
|\nabla u_{0, \lambda} (x, t)-\nabla u_{0, \lambda} (0, 0)|
\le C\left(\frac{ |x|+|t|^{1/2}}{r}\right)^{\alpha_p}
\left\{ 
\left(\fint_{Q_r} |\nabla u_{0, \lambda}|^2\right)^{1/2}
+  r\left(\fint_{Q_r} |F|^p \right)^{1/p} \right\}
$$
for any $(x, t)\in Q_{ r/2}$, where $\alpha_p =1-\frac{d+2}{p}$.
Let 
$P(x, t) =e_j (x_j +\e\chi_j^\lambda (x/\e, t/\e^2))$ with $e_j =\frac{\partial u_{0, \lambda}}{\partial x_j} (0, 0)$.
Then
$$
\aligned
& \left(\fint_{Q_{\theta r}}
|\nabla u_{0, \lambda} (x, t) -\nabla \chi^\lambda (x/\e, t/\e^2) \nabla u_{0, \lambda}  (x, t) -\nabla P  (x, t)|^2 \, dx dt
\right)^{1/2} \\
& \quad
\le C \theta^{\alpha_p}
\left\{ 
\left(\fint_{Q_r} |\nabla u_{0, \lambda}|^2\right)^{1/2}
+  r\left(\fint_{Q_r} |F|^p \right)^{1/p} \right\}
\endaligned
$$
for any $(x, t)\in Q_{\theta r}$.
It follows that
$$
\aligned
 & \left(\fint_{Q_{\theta r}} |\nabla (u_{\e, \lambda} -P) |^2 \right)^{1/2} + \theta r \left (\fint_{Q_{\theta r}} |F|^p \right)^{1/p}
 \\
&\le C \left(\fint_{Q_{\theta r}}
|\nabla u_{\e, \lambda} -\nabla u_{0, \lambda} - (\nabla \chi^\lambda)^\e \nabla u_{0, \lambda}|^2 \right)^{1/2}\\
 & \qquad +  C \theta^{\alpha_p}
\left\{ 
\left(\fint_{Q_r} |\nabla u_{0, \lambda}|^2\right)^{1/2}
+  r\left(\fint_{Q_r} |F|^p \right)^{1/p} \right\}
+\theta r \left (\fint_{Q_{\theta r}} |F|^p \right)^{1/p}
\\
&\le C_0 \left\{ \theta^{-\frac{d+2}{2}}
\left(\frac{ (1+\sqrt{\lambda}) \e}{r}\right)^\sigma 
 +\theta^{\alpha_p} \right\}
\left\{ 
\left(\fint_{Q_{2r}} |\nabla u_{\e, \lambda}|^2\right)^{1/2}
+  2r\left(\fint_{Q_{2r}} |F|^p \right)^{1/p} \right\},
\endaligned
$$
where $C_0$ depends only $d$, $\mu$ and $p$.
Fix $0<\alpha< \alpha_p$.
We choose $\theta\in (0, 1/4)$ so small that
$ C_0 \theta^{\alpha_p} \le (1/2)\theta^{\alpha} $.
With $\theta$ chosen, we assume that $r\ge C_\theta (1+\sqrt{\lambda}) \e$, where $C_\theta>1$ is so large that
$$
C_0 \theta^{-\frac{d+2}{2}} C_\theta^{-\sigma}  <(1/2) \theta^{\alpha}.
$$
This leads to
$$
\aligned
&  \left(\fint_{Q_{\theta r}} |\nabla ( u_{\e, \lambda} -P) |^2 \right)^{1/2}
  + \theta r \left(\fint_{Q_{\theta r}} |F|^p\right)^{1/p}\\
& \qquad
 \le \theta^{\alpha}
 \left\{
   \left(\fint_{Q_{2 r}} |\nabla u_{\e, \lambda} |^2 \right)^{1/2}
  + 2 r \left(\fint_{Q_{ 2r}} |F|^p\right)^{1/p}\right\}.
\endaligned
$$
Since $(\partial_t +\mathcal{L}_{\e, \lambda}) P=0$ in $\mathbb{R}^{d+1}$
 for any $P\in P_{1, \e}^ \lambda$,
we obtain 
\begin{equation}\label{4.10}
\aligned
&  \inf_{P\in P_{1, \e}^\lambda}
\left(\fint_{Q_{\theta r}} |\nabla (u_{\e, \lambda} -P) |^2 \right)^{1/2}
  + \theta r \left(\fint_{Q_{\theta r}} |F|^p\right)^{1/p}\\
& \qquad
 \le \theta^{\alpha}
 \left\{
  \inf_{P\in P_{1, \e}^\lambda }
   \left(\fint_{Q_{2 r}} |\nabla ( u_{\e, \lambda}-P) |^2 \right)^{1/2}
  + 2 r \left(\fint_{Q_{ 2r}} |F|^p\right)^{1/p}\right\}.
\endaligned
\end{equation}
for any $C_\theta (1+\sqrt{\lambda})\e  \le  r<1$.
By an iteration argument it follows that 
\begin{equation}\label{4.11}
\aligned
&  \inf_{P\in P_{1, \e}^\lambda}
\left(\fint_{Q_{ r}} |\nabla (u_{\e, \lambda} -P) |^2 \right)^{1/2}
  +  r \left(\fint_{Q_{r}} |F|^p\right)^{1/p}\\
& \qquad
 \le  C r^\alpha 
 \left\{
  \inf_{P\in P_{1, \e}^\lambda }
   \left(\fint_{Q_{2 }} |\nabla ( u_{\e, \lambda}-P) |^2 \right)^{1/2}
  +  \left(\fint_{Q_{ 2}} |F|^p\right)^{1/p}\right\}.
\endaligned
\end{equation}
for any $(1+\sqrt{\lambda}) \e \le  r< 1$.
This gives the large-scale $C^{1, \alpha}$ estimate (\ref{4.1-0}).
\end{proof}

\begin{thm}[Lipschitz estimate]\label{L-e-10}
Suppose $A$ satisfies conditions (\ref{ellipticity}) and (\ref{periodicity}).
Let $u_{\e, \lambda} $ be a weak solution of $(\partial_t +\mathcal{L}_{\e, \lambda}) u_{\e, \lambda}=F$
in $Q_R$, where $R> (1+\sqrt{\lambda}) \e$ and $F\in L^p(Q_R)$ for some $p>d+2$.
Then, for any $(1+\sqrt{\lambda})\e \le r<  R$,
\begin{equation}\label{4.1-1}
\left(\fint_{Q_r}
|\nabla u_{\e, \lambda} |^2 \right)^{1/2}
\le C 
\left\{ \left(\fint_{Q_R} |\nabla u_{\e, \lambda}|^2 \right)^{1/2}
+ R \left(\fint_{Q_R} |F|^p \right)^{1/p} \right\},
\end{equation}
where $C>0$  depend only on $d$, $\mu$  and $p$.
\end{thm}

\begin{proof}
By translation and dilation we may assume that $R=2$ and
$Q_2=B(0, 2)\times (-4, 0)$.
Define
$$
h(r)=\left(\fint_{Q_r} |\nabla H_r|^2\right)^{1/2},
$$
where $H_r=E_r \cdot (x +\e \chi^\lambda (x/\e, t/\e^2) )$, with $E_r \in \mathbb{R}^d$,
 is a function in $P_{1, \e}^\lambda$ such that
$$
\left(\fint_{Q_r} |\nabla  ( u_{\e, \lambda} - H_r) |^2\right)^{1/2}
=\inf_{P\in P_{1, \e}^\lambda}
\left(\fint_{Q_r} |\nabla (u_{\e, \lambda} -P)|^2 \right)^{1/2}.
$$
Let $C (1+\sqrt{\lambda})\e<  r<1/2$.
Note that
$$
\aligned
|E_{2 r} -E_r| 
&\le \frac{C}{ r}
\inf_{\beta \in \mathbb{R}}
\left(\fint_{Q_{r/2} }
| (E_{2r} -E_r)\cdot x -\beta|^2 \right)^{1/2}\\
&\le 
 \frac{C}{ r}
\inf_{\beta\in \mathbb{R}}
\left(\fint_{Q_{ r/2} }
| H_{2 r} -H_r  -\beta|^2 \right)^{1/2}
+ C |E_{2r} -E_r|  r^{-1} \e,
\endaligned
$$
where $C$ depends only on $d$ and $\mu$.
It follows that if $r\ge C_1 \e$ and $C_1>1$ is sufficiently large, then
\begin{equation}
\aligned
|E_{2r} -E_r| 
&\le 
\frac{C}{ r}
\inf_{\beta\in \mathbb{R}}
\left(\fint_{Q_{ r/2} }
| H_{2 r} -H_r  -\beta|^2 \right)^{1/2}\\
& \le
C \left(\fint_{Q_{ r}}
|\nabla ( H_{2 r} -H_r)|^2 \right)^{1/2},
\endaligned
\end{equation}
where we have used the fact that
 $(\partial_t + \mathcal{L}_{\e, \lambda} )(H_{2 r} -H_r -\beta)=0$ in $\mathbb{R}^{d+1}$
for the last inequality.
Hence,
$$
\aligned
|E_{2r} -E_r|
&\le 
C 
\left(\fint_{Q_{2r}}
|\nabla (u_{\e, \lambda} -H_{2r})|^2 \right)^{1/2}
+
C
\left(\fint_{Q_{r}}
|\nabla (u_{\e, \lambda} -H_r)|^2 \right)^{1/2}\\
&
 \le  C r^\alpha 
 \left\{
  \inf_{P\in P_{1, \e}^\lambda }
   \left(\fint_{Q_{2 }} |\nabla ( u_{\e, \lambda} -P) |^2 \right)^{1/2}
  +  \left(\fint_{Q_{ 2}} |F|^p\right)^{1/p}\right\},
  \endaligned
  $$
  where we have used (\ref{4.11}) for the last step.
  By a simple summation this yields
  $$
h(r) \le C   |E_r| \le  C 
 \left\{
   \left(\fint_{Q_{2 }} |\nabla  u_{\e, \lambda}  |^2 \right)^{1/2}
  +  \left(\fint_{Q_{ 2}} |F|^p\right)^{1/p}\right\},
  $$
which, together with (\ref{4.1-0}),
gives   the large-scale Lipschitz estimate (\ref{4.1-1}).
\end{proof}

\begin{proof}[\bf Proof of Theorem \ref{m-theorem-1}]
Recall that if $\lambda=\e^{k-2}$, then $\mathcal{L}_{\e, \lambda}=\mathcal{L}_\e$.
Also note that in this case, $(1+\sqrt{\lambda}) \e =\e +\e^{k/2}$.
As a result, Theorem \ref{m-theorem-1} follows directly from Theorem \ref{L-e-10}.
\end{proof}

\begin{remark}[$C^{1, \alpha}$  estimate] \label{re-4}
Let $u_\e$ be a weak solution of $(\partial_t +\mathcal{L}_\e) u_\e=F$ in $Q_R$,
where  $R> \e +\e^{k/2}$ and  $F\in L^p(Q_R)$ for some $p>d+2$.
It  follows from Theorem \ref{theorem-4.1} that for $ \e +\e^{k/2} \le  r< R$ and $0< \alpha< 1-\frac{d+2}{p}$,
\begin{equation}\label{C-a}
\aligned
 &\inf_{E\in \mathbb{R}^d}
\left(\fint_{Q_r}
|\nabla u_\e - E - E\nabla \chi^\lambda (x/\e, t/\e^2)|^2 \right)^{1/2}\\
&
 \le C \left(\frac{r}{R}\right)^\alpha
\left\{ \inf_{E\in \mathbb{R}^d}
\left(\fint_{Q_R} |\nabla u_\e -E -E \nabla \chi^\lambda (x/\e, t/\e^2)|^2  \right)^{1/2}
+ R \left(\fint_{Q_R} |F|^p \right)^{1/p} \right\},
\endaligned
\end{equation}
where $\lambda=\e^{k-2}$ and $C$ depends only on $d$, $\mu$, $p$ and $\alpha$.
Note that $\nabla \chi^\lambda (x/\e, t/\e^2)$ is
$\e$-periodic in $x$ and $\e^k$-periodic in $t$.
One may regard (\ref{C-a})  as a $C^{1, \alpha}$ excess-decay estimate for the operator 
$\partial_t +\mathcal{L}_\e$ in (\ref{operator-0}).

Let $E_r\in \mathbb{R}^d$ be the constant for which the left-hand side of (\ref{C-a})
obtains its minimum.
It follows from the proof of Theorem \ref{L-e-10} that
\begin{equation}\label{re-41}
|E_r|
\le C
\left\{ \left(\fint_{Q_R} |\nabla u_\e|^2 \right)^{1/2}
+ R \left(\fint_{Q_R} |F|^p \right)^{1/p} \right\}.
\end{equation}
Let $\chi^\infty$ be defined by (\ref{2.8}).
In view of (\ref{Q-e-5}) we  have
\begin{equation}\label{re-42}
\left(\fint_{Q_r}
|\nabla \chi^\lambda (x/\e, t/\e^2) -\nabla \chi^\infty (x/\e, t/\e^k)|^2 dx dt \right)^{1/2}
\le C \e^{2-k} \|\partial_s A\|_\infty.
\end{equation}
This,  together with (\ref{C-a}) and (\ref{re-41}), yields
\begin{equation}\label{re-43}
\aligned
 &\inf_{E\in \mathbb{R}^d}
\left(\fint_{Q_r}
|\nabla u_\e - E - E\nabla \chi^\infty (x/\e, t/\e^k)|^2\, dx dt \right)^{1/2}\\
&
 \le C \left\{
 \left(\frac{r}{R}\right)^\alpha
 +\e^{2-k} \| \partial_s A\|_\infty  \right\} 
\left\{ \left(\fint_{Q_R} |\nabla u_\e|^2 \right)^{1/2}
+ R \left(\fint_{Q_R} |F|^p \right)^{1/p} \right\},
\endaligned
\end{equation}
for $0<k<2$.
Similarly, for $2<k<\infty$, we obtain
$$
\aligned
 &\inf_{E\in \mathbb{R}^d}
\left(\fint_{Q_r}
|\nabla u_\e - E - E\nabla \chi^0 (x/\e, t/\e^k)|^2\, dx dt \right)^{1/2}\\
&
 \le C \left\{
 \left(\frac{r}{R}\right)^\alpha
 +\e^{k-2} \| \nabla^2 A\|_\infty
 +\e^{k-2} \|\nabla A\|^2_\infty   \right\} 
\left\{ \left(\fint_{Q_R} |\nabla u_\e|^2 \right)^{1/2}
+ R \left(\fint_{Q_R} |F|^p \right)^{1/p} \right\}.
\endaligned
$$
\end{remark}



\section{Higher-order correctors and $C^{2, \alpha}$  estimates} \label{section-5}

In this section we introduce the second-order correctors and 
establish the large-scale $C^{2, \alpha}$ estimates for 
$\mathcal{L}_{\e, \lambda}$.

Let $A_\lambda =\big( a_{ij}^\lambda\big)$ and $B_\lambda =\big(b_{k\ell}^\lambda\big) $ be  the $(1, \lambda)$-periodic  matrices  given by 
(\ref{A-Lambda}) and (\ref{B-3}), respectively.
For $1\le k, \ell \le d$,
the second-order corrector $\chi_{k\ell}^\lambda=\chi_{k\ell}^\lambda (y, s)$ 
is defined to be the weak solution of the cell problem:
\begin{equation}\label{chi-2}
\left\{
\aligned
& \partial_s \chi_{k\ell}^\lambda
-\text{\rm div} \big(A_\lambda \nabla \chi_{k\ell}^\lambda\big)
=b_{k\ell}^\lambda + b_{\ell k}^\lambda
+\frac{\partial}{\partial y_i} \big( a_{i \ell}^\lambda \chi_k^\lambda \big)
+\frac{\partial}{\partial y_i} \big( a^\lambda_{ik} \chi^\lambda_\ell \big) \quad \text{ in } \mathbb{R}^{d+1},\\
 &  \chi_{k\ell}^\lambda \text{ is } (1, \lambda)\text{-periodic in } (y, s),\\
& \int_0^\lambda\!\!\!  \int_{\mathbb{T}^d} \chi_{k\ell}^\lambda\, dy ds =0, 
\endaligned
\right.
\end{equation}
where $(\chi_j^\lambda) $ are the first-order correctors defined by (\ref{cell-lambda}).
Since
$$
\int_0^\lambda \!\!\! \int_{\mathbb{T}^d} b^\lambda_{k\ell}\, dy ds=0,
$$
the solution to (\ref{chi-2}) exists and is unique.
Also, observe that $\chi^\lambda_{k\ell}=\chi_{\ell k}^\lambda$.
Moreover, by the energy estimates,
\begin{equation}\label{5.1}
\fint_0^\lambda \!\!\!\int_{\mathbb{T}^d}
|\nabla \chi^\lambda_{k\ell}|^2
\le C,
\end{equation}
where $C$ depends only on $d$ and $\mu$.

\begin{lemma}\label{lemma-5.1}
Let
$$
u(y, s)=y_k y_\ell + y_k \chi^\lambda_\ell  (y, s) 
+ y_\ell  \chi_k^\lambda  (y, s)  +\chi_{k\ell}^\lambda (y, s).
$$
Then 
$$
\big( \partial_s -\text{\rm div} (A_\lambda\nabla ) \big) u
= \big( \partial_s -\text{\rm div} (\widehat{A_\lambda} \nabla ) \big) (y_k y_\ell)
=- \widehat{a_{\ell k}^\lambda}
-\widehat{a_{k\ell}^\lambda}
$$
in $\mathbb{R}^{d+1}$, where $\widehat{A_\lambda} = \big( \widehat{a^\lambda_{k\ell}} \big)$.
\end{lemma}

\begin{proof}
This follows from a direct computation, using the definitions of $\chi_j^\lambda$ and $\chi_{k\ell}^\lambda$.
\end{proof}

Let $P_0 (x, t)=\beta + e_0 t + e_k x_k +e_{k\ell } x_k x_\ell$ and
\begin{equation}\label{P-5}
\aligned
P_\e (x, t) & =\beta+ e_0 t  + e_k \big\{ x_k  + \e \chi_k^\lambda (x/\e, t/\e^2) \big\} \\
&\quad  +e_{k \ell} \Big\{  x_k x_\ell
+\e x_k \chi_\ell^\lambda (x/\e, t/\e^2) 
+\e x_\ell \chi_k^\lambda (x/\e, t/\e^2)
+\e^2 \chi_{k \ell} (x/\e, t/\e^2)  \Big\},
\endaligned
\end{equation}
where $ \beta, e_0, e_k, e_{k\ell} =e_{\ell k} \in \mathbb{R}$.
It follows from Lemma \ref{lemma-5.1} by rescaling that
$$
(\partial_t +\mathcal{L}_{\e, \lambda} ) P_\e 
=(\partial_t +\mathcal{L}_{0, \lambda} ) P_0
=e_0-2 e_{k\ell} \widehat{a^\lambda_{k\ell}} \quad \text{ in } \mathbb{R}^{d+1}.
$$
We shall use $P^\lambda_{2, \e}$ to denote the 
set of all functions $P_\e(x, t)$ in the form of (\ref{P-5}) such that
$(\partial_t +\mathcal{L}_{\e, \lambda}) P_\e =0$.
Let $C^\sigma_p (Q_R)$ denote the space of  H\"older continuous functions $u=u(x, t)$ such that
$$
\| u\|_{C^\sigma (Q_R)}
:=R^\sigma \sup
\left\{ \frac{| u(x, t)-u(y, s)|}{ ( |x-y| +|t-s|^{1/2} )^\sigma}:
 (x, t), (y, s) \in Q_R \text{ and } (x, t)\neq (y, s) \right\}<\infty,
 $$
where $\sigma \in (0, 1)$.

\begin{thm}[$C^{2, \alpha}$ estimate]\label{C-5}
Suppose  $A$ satisfies conditions (\ref{ellipticity}) and (\ref{periodicity}).
Let $u_{\e, \lambda}$ be a weak solution of
$(\partial_t +\mathcal{L}_{\e, \lambda}) u_{\e, \lambda} =F$ in $Q_R$,
where $R> (1+\sqrt{\lambda})\e $ and $F\in C^{\sigma}(Q_R)$ for some $\sigma\in (0, 1)$.
Then, for any $(1+\sqrt{\lambda}) \e \le r <R$ and $0<\alpha<\sigma$,
\begin{equation}\label{5.3-0}
\aligned
 &\inf_{P\in P_{2,\e}^\lambda}
\left(\fint_{Q_r} |\nabla (u_{\e, \lambda}  -P)|^2 \right)^{1/2}\\
&\qquad
\le C \left(\frac{r}{R}\right)^{1+\alpha} 
\left\{ \inf_{P\in P_{2, \e}^\lambda}
\left(\fint_{Q_R} |\nabla ( u_{\e, \lambda} -P) |^2\right)^{1/2}
+ R \| F\|_{C^\sigma (Q_R)}
\right\},
\endaligned
\end{equation}
 where
$C$ depends only on $d$, $\sigma$, $\mu$,  and $\alpha$.
\end{thm}

\begin{proof}
By translation and dilation we may assume that
$R=2$ and $Q_2=B(0, 2)\times (-4, 0)$.
By subtracting $e_0 t$ from $u_{\e, \lambda} $, we may also assume that $F(0, 0)=0$, which implies 
$\|F\|_{L^\infty(Q_r)} \le C   \| F\|_{C^\sigma (Q_r)}$.
Let $(1+\sqrt{\lambda}) \e < \theta r < r<1$,
where $\theta\in (0, 1/4)$ is to be chosen later.
Let $u_{0, \lambda}$ be the weak solution of
$(\partial_t +\mathcal{L}_{0, \lambda}) u_{0, \lambda}=F$ in $Q_r$, given by Theorem \ref{theorem-3.1}.
By the classical $C^{2+\alpha}$ estimates for parabolic systems with constant coefficients,
\begin{equation}\label{5.3-1}
\aligned
& \Big|\frac{\partial u_{0, \lambda}}{\partial x_i} (x, t) -\frac{\partial u_{0, \lambda}}{\partial x_i} (0, 0)
-\frac{\partial^2 u_{0, \lambda}}{\partial x_j \partial x_i} (0, 0) x_j \Big|\\
&\le \Big| \frac{\partial u_{0, \lambda}}{\partial x_i} (x, t) - \frac{\partial u_{0, \lambda}}{\partial x_i} (x, 0)\Big|
+\Big|\frac{\partial u_{0, \lambda}}{\partial x_i} (x, 0) -\frac{\partial u_{0, \lambda}}{\partial x_i} (0, 0)
-\frac{\partial^2 u_{0, \lambda}}{\partial x_j \partial x_i} (0, 0) x_j \Big|\\
& \le C \theta^{1+\sigma}
\left\{ \left(\fint_{Q_r} | \nabla u_{0, \lambda}|^2 \right)^{1/2}
+r \| F \|_{C^\sigma (Q_r)} \right\}\\
 &\le C \theta^{1+\sigma}
\left\{ \left(\fint_{Q_{2r}} | \nabla u_{\e, \lambda}|^2 \right)^{1/2}
+r \| F \|_{C^\sigma (Q_r)} \right\}
\endaligned
\end{equation}
for any $(x, t)\in Q_{\theta r}$, 
where we have used (\ref{3.3-1}) for the last inequality.
Let $P_0 (x, t)=e_0 t +e_ix_i + e_{ij} x_i x_j$, where
\begin{equation}\label{e-0}
e_0 =\partial_t u_{0, \lambda} (0, 0),\ 
e_i =\frac{\partial u_{0, \lambda} }{\partial  x_i} (0, 0), \text{ and } \
e_{ij}=\frac12 \frac{\partial^2 u_{0, \lambda}}{\partial x_i \partial x_j} (0, 0).
\end{equation}
Note that
\begin{equation}\label{5.3-2}
(\partial_t +\mathcal{L}_{0, \lambda} ) P_0
  = e_0 -2 e_{ij} \widehat{a_{ij}^\lambda}
 =(\partial_t +\mathcal{L}_{0, \lambda}) u_0 (0, 0)=F(0, 0)=0,
 \end{equation}
and by (\ref{5.3-1}),
\begin{equation}\label{5.3-4}
\|\nabla (u_{0, \lambda} -P_0)\|_{L^\infty(Q_{\theta r})}
\le C \theta^{1+\sigma}
\left\{ \left(\fint_{Q_r} | \nabla u_{0, \lambda}|^2 \right)^{1/2}
+r \| F \|_{C^\sigma (Q_r)} \right\}.
\end{equation}
This, together with the inequality  (\ref{3.3}), gives
\begin{equation}\label{5.3-5}
\aligned
& \left(\fint_{Q_{\theta r}}
|\nabla u_{\e, \lambda} -\nabla P_0 -(\nabla \chi^\lambda)^\e (\nabla P_0) |^2 \right)^{1/2}\\
& \quad
\le 
C \left\{ \theta^{1+\sigma}
+ \left(\frac{(1+\sqrt{\lambda}}{r} \right)^\sigma \right\}
\left\{ 
 \left(\fint_{Q_{2r}} |\nabla u_{\e, \lambda}|^2 \right)^{1/2}
+ r  \| F\|_{C^\sigma(Q_{2r})}
 \right\}.
\endaligned
\end{equation}

Let $P_\e= P_\e (x, t)$ be given by (\ref{P-5}) with the same coefficients as those of $P_0$ in (\ref{e-0}).
Then $(\partial_t +\mathcal{L}_{\e, \lambda}) P_\e =(\partial_t +\mathcal{L}_{0, \lambda}) P_0=0$, and
\begin{equation}\label{5.3-7}
|\nabla P_\e -\nabla P_0 -(\nabla \chi^\lambda)^\e (\nabla P_0)|
\le \e |e_{k\ell} \nabla \chi^\lambda_{k\ell} (x/\e, t/\e^2)|.
\end{equation}
In view of (\ref{5.3-5}) we obtain 
\begin{equation}\label{5.3-6}
\aligned
& \left(\fint_{Q_{\theta r}}
|\nabla (u_{\e, \lambda} -P_\e)|^2 \right)^{1/2}\\
& \quad
\le 
C \left\{ \theta^{1+\sigma}
+ \left(\frac{(1+\sqrt{\lambda})\e }{r} \right)^\sigma 
\right\}
\left\{ 
 \left(\fint_{Q_{2r}} |\nabla u_{\e, \lambda}|^2 \right)^{1/2}
+ r \| F\|_{C^\sigma (Q_{2r})}
 \right\},
\endaligned
\end{equation}
where we have used (\ref{5.1}) and  the assumption that $\theta r \ge (1+\sqrt{\lambda} )\e$.

To proceed, we let
$$
\Psi (r) =\inf_{P\in P_{2, \e}^\lambda}
\left(\fint_{Q_r} |\nabla (u_{\e, \lambda} -P) |^2 \right)^{1/2}
+ r\| F\|_{C^\sigma (Q_r)}.
$$
It follows from (\ref{5.3-6})  that
$$
\Psi (\theta r)
\le 
C_0 \left\{ \theta^{1+\sigma}
+ \left(\frac{(1+\sqrt{\lambda})\e }{r} \right)^\sigma 
\right\} \Psi (2r)
$$
for $(1+\sqrt{\lambda}) \e < \theta r< r<1$,
where $C_0$ depends only on $d$, $\mu$ and $\sigma$.
Fix $\alpha \in (0, \sigma)$.
Choose $\theta \in (0, 1/4)$ so small that
$C_0\theta^{1+\sigma} \le (1/2) (\theta/2)^{1+\alpha}$.
With $\theta$ chosen, we may choose $C_1>1$
 so large that $ C_0 C_1^{-\sigma} \le (1/2) (\theta/2) ^{1+\alpha}$.
As a result, for $C_1 (1+\sqrt{\lambda}) \e < \theta r< r<1$, we have
$$
\Psi (\theta r) \le (\theta/2)^{1+\alpha} \Psi (2r).
$$
By a simple iteration  argument  this gives $\Psi (r)\le C r^{1+\alpha} \Psi (2)$ for any $(1+\sqrt{\lambda})  \e \le  r<  2$.
\end{proof}

\begin{remark}[Liouville property]
By letting  $\lambda=\e^{k-2}$ in 
Theorem \ref{C-5}  we obtain a $C^{2, \alpha}$ excess-decay estimate for 
$\partial_t +\mathcal{L}_\e$ in (\ref{operator-0}) for any $0<k<\infty$.
The estimate may be used 
to establish  a Liouville property  for the operator.
Indeed, let  $u_\e$ be a solution of $(\partial_t +\mathcal{L}_\e)u_\e=0$ in 
$\mathbb{R}^d \times (-\infty, t_0)$ for some $t_0\in \mathbb{R}$.
Suppose there exist $C_u>0$ and $\alpha \in (0, 1)$ such that 
\begin{equation}\label{Li-1}
\left(\fint_{Q_R (0, t_0)} |u_\e |^2 \right)^{1/2}
\le C_u R^{2+\alpha}
\end{equation}
for any $R>1$. By Cacciopoli's inequality it follows that
$$
\left(\fint_{Q_R (0, t_0)} |\nabla u_\e |^2 \right)^{1/2}
\le C R^{1+\alpha}
$$
for any $R>1$. This, together with (\ref{5.3-0}), implies that
$u_\e =P$ 
in $\mathbb{R}^d \times (-\infty, t_0)$ for some $P\in P_{2, \e}^\lambda $.
\end{remark}



\section{Boundary Lipschitz estimates}\label{section-7}

In this section we establish large-scale boundary Lipschitz estimates for the operator
$\partial_t +\mathcal{L}_{\e, \lambda}$,
where $\mathcal{L}_{\e, \lambda} =-\text{\rm div} \big( A_\lambda (x/\e, t/\e^2)\nabla \big)$.
As a consequence, we obtain the large-scale boundary  Lipschitz estimate
for $\partial_t +\mathcal{L}_\e$ in  Theorem \ref{m-theorem-3}.

Throughout this section we will assume that $\Omega$ is a bounded $C^{1, \alpha}$ domain for some $\alpha \in (0, 1)$. Let 
\begin{equation}\label{D}
\aligned
D_r (x_0, t_0)  &
= \big( B(x_0, r)\cap \Omega \big) \times (t_0-r^2, t_0),\\
\Delta_r (x_0, t_0) & = \big( B(x_0, r)\cap \partial\Omega\big) \times (t_0-r^2, t_0),
\endaligned
\end{equation}
where $x_0\in \partial\Omega$ and  $t_0\in \mathbb{R}$.
For $ \alpha\in (0, 1)$ and $\Delta_r=\Delta_r (x_0, t_0)$,
we use $C^{1+\alpha} (\Delta_r)$ to denote the parabolic $C^{1+\alpha}$ space of functions on $\Delta_r$ with 
the scale-invariant norm, 
$$
\| f\|_{C^{1+\alpha}(\Delta_r)}
:=\| f\|_{L^\infty (\Delta_r)}
+ r\|\nabla_{\tan} f\|_{L^\infty(\Delta_r)}
+ r
\|\nabla_{\tan} f \|_{C^\alpha (\Delta_r)}
+ 
\| f\|_{C_t^{\frac{1+\alpha}{2}} (\Delta_r)},
$$
where $\| g\|_{C^\alpha(\Delta_r)}$ is the smallest constant $C_0$ such that
$$
| g(x, t)-g(y, s)|\le C_0 r^{-\alpha} (|x-y| +|t-s|^{1/2})^\alpha 
$$
for any $(x, t), (y, s)\in \Delta_r$, and
$$
\| f\|_{C_t^{\frac{1+\alpha}{2}} (\Delta_r)}
=\inf \left\{ C: \ |f(y, \tau)-f(y, s)|\le C r^{-1-\alpha}  |\tau-s|^{\frac{1+\alpha}{2}} \text{ for any } (y, \tau ), 
(y, s)\in \Delta_r \right\}.
$$

\begin{thm}\label{thm-7.1}
Assume $A=A(y,s)$ satisfies (\ref{ellipticity}) and (\ref{periodicity}).
Suppose that  $(\partial_t +\mathcal{L}_{\e, \lambda}) u_{\e, \lambda} =F$ in $D_R=D_R (x_0, t_0)$
and $u_{\e, \lambda} =f$ on $\Delta_R= \Delta_R (x_0, t_0)$,
where $x_0\in \partial\Omega$,  $(1+\sqrt{\lambda} ) \e <R\le 1$, and
$F\in L^p(D_R)$ for some $p>d+2$.
Then, for any $(1+\sqrt{\lambda}) \e\le  r< R$,
\begin{equation}\label{7.1-0}
\aligned
& \left(\fint_{D_r} |\nabla u_{\e, \lambda} |^2 \right)^{1/2}\\
& \le C \left\{ \left(\fint_{D_R}
|\nabla u_{\e, \lambda}|^2 \right)^{1/2}
+ R^{-1} \| f\|_{C^{1+\alpha}(\Delta_R)}
+ R \left(\fint_{D_R} |F|^p \right)^{1/p} \right\},
\endaligned
\end{equation}
where $C$ depends only on $d$, $\mu$, $p$, $\alpha$, and $\Omega$.
\end{thm}

To prove Theorem \ref{thm-7.1}, 
we localize the boundary of $\Omega$.
Let $\psi: \mathbb{R}^{d-1} \to \mathbb{R}$ be a $C^{1, \alpha}$ function such that
 $\psi (0)=0$ and $\|\psi\|_{C^{1, \alpha} (\mathbb{R}^{d-1})} \le M$.
 Define
 \begin{equation}\label{T}
 \aligned
 T_r &  =\big\{
 (x^\prime, x_d):
 |x^\prime|< r \text{ and } \psi (x^\prime)
 < x_d< 100\sqrt{d} (M+1) \big\} \times (-r^2, 0), \\
I_r &  =\big\{
 (x^\prime, \psi (x^\prime)):
 |x^\prime|< r   \big\} \times (-r^2, 0),
\endaligned
\end{equation} 
where $0<r<\infty$

We begin with an approximation lemma.

\begin{lemma}\label{lemma-7.2}
Assume $A$ satisfies (\ref{ellipticity}) and (\ref{periodicity}).
Suppose that $(\partial_t +\mathcal{L}_{\e, \lambda}) u_{\e, \lambda}=F$
in $T_{2r}$ and $u_{\e, \lambda}=f$ on $I_{2r}$ for some $0<r\le 1$.
Then there exists a function  $u_{0, \lambda}$ such that
$(\partial_t +\mathcal{L}_{0, \lambda}) u_{0, \lambda} =F$ in $T_r$,
$u_{0, \lambda}=f$ on $I_r$, and
\begin{equation}\label{7.2-0}
\aligned
& \left(\fint_{T_r}
| u_{\e, \lambda} -u_{0, \lambda} |^2\right)^{1/2}\\
&
\le C \left( \frac{(1+\sqrt{\lambda}) \e}{r} \right)^\sigma
\left\{\left(\fint_{T_{2r}}
|u_{\e, \lambda}|^2 \right)^{1/2}
+  \| f\|_{C^{1+\alpha} (I_{2r})}
+ r^2 \left(\fint_{T_{2r}} |F|^2 \right)^{1/2} \right\},
\endaligned
\end{equation}
where $\sigma\in (0, 1)$ and $C>0$ depend only on $d$, $\mu$,  $p$, and $M$.
\end{lemma}

\begin{proof}
The proof is similar to that of Theorem \ref{theorem-3.1}.
By dilation we may assume $r=1$.
Let $u_{0, \lambda}$ be the weak solution to the initial-Dirichlet problem,
$$
(\partial_t +\mathcal{L}_{0, \lambda}) u_{0, \lambda} 
=F \quad \text{ in } T_1 \quad 
\text{ and } \quad
u_{0, \lambda} =u_{\e, \lambda} \quad
\text{ on } \partial_p T_1.
$$
It follows by the Meyers-type estimates and Cacciopoli's inequality for parabolic systems that
\begin{equation}\label{7.2-1}
\aligned
& \left(\fint_{T_1} |\nabla u_{0, \lambda}|^q\right)^{1/q}
 \le  C \left(\fint_{T_1} |u_{\e, \lambda}|^q \right)^{1/q}\\
&\le C \left\{
\left( \fint_{T_2} |u_{\e, \lambda} |^2 \right)^{1/2}
+ \left(\fint_{T_2} |F|^2\right)^{1/2}
+ \| f \|_{C^{1+\alpha}(I_2)}
\right\},
\endaligned
\end{equation}
where $q>2$ and $C>0$ depend only on $d$, $\mu$, $\alpha$ and $M$.
To see (\ref{7.2-0}),
we define $w_{\e}$ as in (\ref{w-3}).
Using the same argument as in the proof of Theorem \ref{theorem-3.1},
we may show that
\begin{equation}\label{7.2-2}
\left(\fint_{T_1} |\nabla w_\e|^2 \right)^{1/2}
\le C \delta^\sigma \left(\fint_{T_1} |\nabla u_{0, \lambda}|^q \right)^{1/q},
\end{equation}
where $\delta =(1+\sqrt{\lambda} ) \e$ and $\sigma =\frac12 -\frac{1}{q}>0$.
Since $w_\e=0$ on $\partial_p T_1$,
it follows from Poincar\'e's inequality and (\ref{7.2-1}) that
\begin{equation}\label{7.2-3}
\aligned
&\left(\fint_{T_1} |w_\e|^2 \right)^{1/2}\\
& \le C \delta^\sigma 
 \left\{
\left( \fint_{T_2} |u_{\e, \lambda} |^2 \right)^{1/2}
+ \left(\fint_{T_2} |F|^2\right)^{1/2}
+ \| f \|_{C^{1+\alpha} (I_2)}
\right\}.
\endaligned
\end{equation}
This yields (\ref{7.2-0}),
as $\| w_\e -(u_{\e, \lambda} -u_{0, \lambda}) \|_{L^2(T_1)}$
is also bounded by the right-hand side of (\ref{7.2-3}).
\end{proof}

For  a function $u$  in $T_r$, define
\begin{equation}\label{Psi}
\Psi (r; u)
=\frac{1}{r} \inf_{\substack{{ E\in \mathbb{R}^d} \\ \beta\in \mathbb{R}}}
\left\{
\left(\fint_{T_r}
|u- E\cdot x -\beta |^2 \right)^{1/2}
+ \| u -E\cdot x -\beta \|_{C^{1+\alpha} (I_r)}\right\}.
\end{equation}

\begin{lemma}\label{lemma-7.3}
Suppose that $(\partial_t +\mathcal{L}_{0, \lambda}) u=F$ in $T_r$,
where $0<r\le 1$ and $F\in L^p(T_r)$ for some $p>d+2$.
Then there exists $\theta \in (0, 1/4)$, depending only on 
$d$, $\mu$, $\alpha$,  $p$, and $M$, such that
\begin{equation}\label{7.3-0}
\Psi (\theta r; u ) + \theta r \left(\fint_{T_{\theta r}} |F|^p\right)^{1/p}
 \le \frac12 
 \left\{ \Psi (r; u)
 + r \left(\fint_{T_r} |F|^p \right)^{1/p} \right\}.
\end{equation}
\end{lemma}

\begin{proof}
Choose $\sigma \in (0, 1)$ such that $\sigma <\min (\alpha, 1-\frac{d+2}{p})$.
The proof uses the boundary $C^{1+ \sigma}$ estimate for second-order parabolic
systems with constant coefficients in $C^{1, \alpha}$ cylinders.
Let $E_0=\nabla u (0, 0)$ and $\beta_0 =u (0, 0)$.
Then, for any $(x, t)\in T_{r/2}$,
$$
\aligned
 &| u(x, t) - E_0\cdot x -\beta_0|\\
& \le C  (|x| + |t|^{1/2} )^{1+\sigma}
\left\{
\left(\fint_{T_{r}} |u|^2 \right)^{1/2}
+ \| u\|_{C^{1+\alpha} (\Delta_r)}
+ r^2 \left(\fint_{T_r} |F|^p\right)^{1/p} \right\},
\endaligned
$$
where $C$ depends only on $d$, $\mu$,  $\alpha$, $p$, and $M$.
It follows that the left-hand side of (\ref{7.3-0}) is bounded by
$$
\frac{C_0 \theta^\sigma}{r}
\left\{
\left(\fint_{T_{r}} |u|^2 \right)^{1/2}
+ \| u\|_{C^{1+\alpha} (\Delta_r)}
+ r^2 \left(\fint_{T_r} |F|^p\right)^{1/p} \right\}.
$$
Since $(\partial_t +\mathcal{L}_{0, \lambda}) (E\cdot x +\beta)=0$ for any $E\in \mathbb{R}^d$ and
$\beta\in  \mathbb{R}$, we may replace $u$ by $u-E\cdot x -\beta$.
As a result,  we see that the left-hand side of (\ref{7.3-0}) is bounded by 
$$
C_0 \theta^\sigma \left\{ \Psi (r; u)
 + r \left(\fint_{T_r} |F|^p \right)^{1/p} \right\}.
 $$
To finish the proof,
we choose $\theta\in (0, 1/4)$ so small that $C_0\theta^\sigma \le (1/2)$.
\end{proof}

\begin{lemma}\label{lemma-7.4}
Suppose that $(\partial_t +\mathcal{L}_{\e, \lambda}) u_{\e, \lambda}=F$ in $T_2$
and $u=f$ on $I_{2}$,
where $(1+\sqrt{\lambda})\e < 1$ and $F\in L^p(T_{2})$ for some $p>d+2$.
Let $\theta\in (0, 1/4) $ be given by Lemma \ref{lemma-7.3}.
Then for any $(1+\sqrt{\lambda}) \e\le  r\le 1$, 
\begin{equation}\label{7.4-0}
\aligned
& \Psi (\theta r; u_{\e, \lambda})
+\theta r \left(\fint_{T_{\theta r}} |F|^p\right)^{1/p}\\
&\le \frac12 \left\{
\Psi (r; u_{\e, \lambda})
+ r \left(\fint_{T_ r} |F|^p\right)^{1/p}\right\}\\
& \quad
+ C \left(\frac{(1+\sqrt{\lambda}) \e}{r}\right)^\sigma
 \left\{ 
 \frac{1}{r}\left(\fint_{T_{2r}} |u_{\e, \lambda}|^2 \right)^{1/2}
+ r \left(\fint_{T_{2r}} |F|^p \right)^{1/p}
+ r^{-1}\| f\|_{C^{1+\alpha} (I_{2r})}
\right\},
\endaligned
\end{equation}
where $C$ depends only on $d$, $\mu$, $p$, $\alpha$ and $M$.
\end{lemma}

\begin{proof}
Fix $(1+\sqrt{\lambda}) \e \le r\le 1$.
Let $u_{0, \lambda}$ be the solution of $(\partial_t +\mathcal{L}_{0, \lambda}) u_{0, \lambda}=F$
in $T_{r}$ with $u_{0, \lambda}=f$ on $I_r$, given by Lemma \ref{lemma-7.2}.
Observe that
$$
\aligned
\Psi (\theta r;  & u_{\e, \lambda})
+ \theta r \left( \fint_{T_{\theta r}} |F|^p\right)^{1/p}\\
 & \le \Psi (\theta r; u_{0, \lambda})  +
 \theta r \left( \fint_{T_{\theta r}} |F|^p\right)^{1/p}
 +\frac{1}{\theta r}
 \left(\fint_{T_{\theta r}} |u_{\e, \lambda} -u_{0, \lambda}|^2 \right)^{1/2}\\
 &\le 
 \frac12 \left\{
  \Psi ( r; u_{0, \lambda}) +
  r \left( \fint_{T_{ r}} |F|^p\right)^{1/p}\right\}
  +\frac{C_\theta}{ r}
 \left(\fint_{T_{ r}} |u_{\e, \lambda} -u_{0, \lambda}|^2 \right)^{1/2}\\
 & \le 
 \frac12 \left\{
  \Psi ( r; u_{\e, \lambda}) +
  r \left( \fint_{T_{ r}} |F|^p\right)^{1/p}\right\}
  +\frac{C_\theta}{ r}
 \left(\fint_{T_{ r}} |u_{\e, \lambda} -u_{0, \lambda}|^2 \right)^{1/2},
\endaligned
$$
where we have used Lemma \ref{lemma-7.3} for the second inequality.
This, together with Lemma \ref{lemma-7.2},
gives (\ref{7.4-0}).
\end{proof}

The proof of the following lemma may be found in \cite{Shen-book}.

\begin{lemma}\label{lemma-7.5}
Let $H(r)$ and $h(r)$ be two nonnegative and continuous functions on the interval 
$[0, 1]$. Let $0<\delta<(1/4)$. Suppose that there exists a constant $C_0$ such that
\begin{equation}\label{7.5-0}
\max_{r\le t\le 2r} H(t) \le C_0 H(2r)
\quad \text{ and } \quad 
\max_{r\le, t, s\le 2r} |h(t) -h(s)|\le C_0 H(2r)
\end{equation}
for any $r\in [\delta, 1/2]$.
Suppose further that
\begin{equation}\label{7.4-1}
H(\theta r) \le \frac 12 H(r)
+C_0 \eta (\delta/ r) 
\Big\{ H(2r) + h(2r) \Big\}
\end{equation}
for any $r\in [\delta, 1/2]$, where $\theta\in (0, 1/4)$ and $\eta (t)$ is a nonnegative and
nondecreasing function on $[0, 1]$ such that $\eta (0)=0$ and
\begin{equation}\label{7.4-2}
\int_0^1 \frac{\eta (t)}{t} \, dt <\infty.
\end{equation}
Then
\begin{equation}\label{7.4-3}
\max_{\delta\le r\le 1} 
\big\{ H(r) + h(r) \big\}
\le C \big\{ H(1) + h(1) \big\},
\end{equation}
where $C$ deepnds only on $C_0$, $\theta$, and the function $\eta (t)$.
\end{lemma}

We are now ready to give the proof of Theorem \ref{thm-7.1}

\begin{proof}[\bf Proof of Theorem \ref{thm-7.1}]
By translation and dilation we may assume that $(x_0, t_0)=(0, 0)$ and
$R=1$. Moreover, it suffices to show that for $(1+\sqrt{\lambda}) \e\le r< 2$,
\begin{equation}\label{7.5-1}
  \left(\fint_{T_r} |\nabla u_{\e, \lambda} |^2 \right)^{1/2}\\
 \le C \left\{ \left(\fint_{T_2}  |\nabla u_{\e, \lambda}|^2 \right)^{1/2}
 + \| f\|_{C^{1+\alpha} (I_2)}
 + \left(\fint_{T_2} |F|^p \right)^{1/p} \right\},
 \end{equation}
 where $(\partial_t + \mathcal{L}_{\e, \lambda}) u_{\e, \lambda} =F$ in $T_2$
 and $u_{\e, \lambda} =f$ on $I_2$.
To this end, we apply Lemma \ref{lemma-7.5} with
$$
H(r)=\Psi (r; u_{\e, \lambda})
+ r \left(\fint_{T_r} |F|^p\right)^{1/p}
$$
and $h(t)=|E_r|$, where $E_r$ is a vector in $\mathbb{R}^d$ such that
$$
\Psi (r; u_{\e, \lambda})
=\frac{1}{r} \inf_{ \beta\in \mathbb{R}}
\left\{
\left(\fint_{T_r}
|u_{\e, \lambda}- E_r\cdot x -\beta |^2 \right)^{1/2}
+ \| f -E_r\cdot x -\beta \|_{C^{1+\alpha} (I_r)}\right\}.
$$
Note that by (\ref{7.4-0}),
$$
H(\theta r)
\le \frac 12 H(r) +C_0 \left(\frac{\delta}{r} \right)^\sigma
\Big\{ H(2r) + h(2r) \Big\}
$$
for $r\in [\delta, 1]$, where $\delta =(1+\sqrt{\lambda}) \e$.
This gives (\ref{7.4-1}) with $\eta (t)=t^\sigma$, which satisfies (\ref{7.4-2}).

It is easy to see that $H(r)$ satisfies the first inequality in (\ref{7.5-0}).
To verify the second, we note that for $r\le t, s \le 2r$,
$$
\aligned
|h(t)-h(s)|
&\le |E_t -E_s|\\
& \le \frac{C}{r}
\inf_{\beta \in \mathbb{R}}
\left(\fint_{T_r} |(E_t-E_s)\cdot x -\beta|^2 \right)^{1/2}\\
&
\le \frac{C}{r}
\inf_{\beta \in \mathbb{R}}
\left(\fint_{T_r} |u_{\e, \lambda} - E_t\cdot x -\beta|^2 \right)^{1/2}
+ \frac{C}{r}
\inf_{\beta \in \mathbb{R}}
\left(\fint_{T_r} |u_{\e, \lambda} - E_s\cdot x -\beta|^2 \right)^{1/2}\\
&\le C \big\{ H(t) +H(s) \big\}\\
&\le C H(2r),
\endaligned
$$
where $C$ depends only on $d$, $\alpha$ and $M$.
Thus, by Lemma \ref{lemma-7.5}, we obtain 
$$
\aligned
 \frac{1}{r}
\inf_{\beta \in \mathbb{R}}
\left(\fint_{T_r} |u_{\e, \lambda}  -\beta|^2 \right)^{1/2}
& \le H(r) + h(r)\\
&\le C \big\{ H(1) + h(1) \big\}\\
&\le C  \left\{ \left(\fint_{T_1} |u_{\e, \lambda} |^2 \right)^{1/2}
+ \| f\|_{C^{1+\alpha} (I_1)}
+ \left(\fint_{T_1} |F|^p \right)^{1/p} \right\}.
\endaligned
$$
By Cacciopoli's inequality for parabolic systems (see e.g. \cite[Appendix]{Armstrong-2018}),
$$
\left(\fint_{T_{r/2}} |\nabla u_{\e, \lambda}|^2 \right)^{1/2}
 \le C  \left\{ \left(\fint_{T_1} |u_{\e, \lambda} |^2 \right)^{1/2}
+ \| f\|_{C^{1+\alpha} (I_1)}
+ \left(\fint_{T_1} |F|^p \right)^{1/p} \right\}.
$$
Since $(\partial_t +\mathcal{L}_{\e, \lambda}) (\beta)=0$ for any 
$\beta \in \mathbb{R}$,
we may replace $u_{\e, \lambda}$ in the right-hand side of the
inequality above by $u_{\e, \lambda}-\beta$.
This, together with Poincar\'e-type inequality for parabolic systems, yields (\ref{7.5-1}).
\end{proof}

\begin{proof}[\bf Proof of Theorem \ref{m-theorem-3}]
Since $\mathcal{L}_\e=\mathcal{L}_{\e, \lambda}$ for $\lambda =\e^{k-2}$,
Theorem \ref{m-theorem-3} follows readily from Theorem \ref{thm-7.1}.
\end{proof}


\section{Convergence rates}\label{section-6}

In this section we investigate the problem of convergence rates for the initial-Dirichlet problem,
\begin{equation}\label{DP}
\left\{
\aligned
(\partial_t +\mathcal{L}_{\e, \lambda} ) u_{\e, \lambda}   & =F & \quad & \text{ in } \Omega_T,\\ 
u_{\e, \lambda}  & =f & \quad & \text{ on } \partial_p \Omega_T,
\endaligned
\right.
\end{equation}
where $\Omega$ is a bounded domain in $\mathbb{R}^d$ and $\Omega_T=\Omega\times  (0, T)$.
As a consequence, we  obtain rates of convergence  for the operator $\partial_t +\mathcal{L}_\e$ in (\ref{operator-0}).

Let $u_{0, \lambda}$ be the solution of the homogenized problem for (\ref{DP}),
\begin{equation}\label{DP-0}
\left\{
\aligned
(\partial_t +\mathcal{L}_{0, \lambda} ) u_{0, \lambda}  & =F & \quad & \text{ in } \Omega_T,\\ 
u_{0, \lambda} & =f & \quad & \text{ on } \partial_p \Omega_T.
\endaligned
\right.
\end{equation}
Let $w_{\e}$ be the two-scale expansion given by (\ref{w-3}).
As before,  the operator $K_\e$ is defined by $K_\e (f) =S_\delta (\eta_\delta f)$ with
$\delta =(1+\sqrt{\lambda} )\e$.
The cut-off function $\eta_\delta=\eta_\delta^1 (x)  \eta_\delta^2 (t) $ is chosen so that 
$0\le \eta_\delta\le 1$, $|\nabla \eta_\delta|\le C /\delta$,
$|\partial_t \eta_\delta | +|\nabla^2 \eta_\delta|\le C /\delta^2$, and
$$
\eta_\delta=1 \quad \text{ in } \Omega_T \setminus \Omega_{T, 3\delta} \quad 
\text{ and } \quad 
\eta_\delta=0 \quad \text{ in } \Omega_{T,  2\delta},
$$
where $\Omega_{T, \rho }$ denotes the (parabolic) boundary layer
\begin{equation}\label{layer}
\Omega_{T, \rho} =
\Big( \big\{ x\in \Omega: \, \text{\rm dist} (x, \partial\Omega)<\rho \big\}
\times (0, T)\Big) 
\cup \Big( \Omega \times (0, \rho^2)\Big)
\end{equation}
for $0< \rho\le c$.

\begin{lemma}\label{lemma-6.0}
Let $\Omega$ be a a bounded Lipschitz domain in $\mathbb{R}^d$.
Let  $\Omega_{T, \rho }$ be defined by (\ref{layer}). Then
\begin{equation}\label{layer-1}
\| \nabla g \|_{L^2(\Omega_{T, \rho})}
\le C \sqrt{ \rho } \, 
\Big\{ \|\nabla g\|_{L^2(\Omega_T)}
+\|\nabla^2 g\|_{L^2(\Omega_T)}
+\| \partial_t g \|_{L^2(\Omega)}
\Big\},
\end{equation}
where $C$ depends only on $d$, $\Omega$ and $T$.
\end{lemma}

\begin{proof}
Let $\Omega_\rho =\big\{ x\in \Omega:  \text{dist} (x, \partial\Omega)< \rho  \big\}$.
Then
$$
\| \nabla g(\cdot, t)  \|_{L^2(\Omega_\rho)}
\le C \sqrt{\rho} \, \| \nabla g(\cdot, t) \|_{H^1(\Omega)}.
$$
It follows that
$$
\|\nabla g\|_{L^2(\Omega_\rho \times (0, T))}
\le C \sqrt{\rho} \, \Big\{ \|\nabla g\|_{L^2 (\Omega_T)}
+\|\nabla^2 g\|_{L^2(\Omega_T)}\Big\}.
$$
To estimate $\| \nabla g\|_{L^2( (\Omega\setminus \Omega_\rho) \times (0, \rho^2))}$,
we choose a cut-off function $\theta\in C_0^\infty(\Omega)$ such that
$0\le \theta\le 1$, $\theta=1$ on $\Omega\setminus \Omega_\rho$, and $|\nabla \theta|\le C/\rho$.
By Fubini's Theorem we may also choose $t_0\in (T/2, T)$ such that
$$
\int_\Omega |\nabla g (x, t_0)|^2\, dx
\le \frac{2}{T} \int_{\Omega_T} |\nabla g|^2 \, dx dt.
$$
Note that for any $t\in (0, \rho^2)$,
$$
\aligned
\int_\Omega |\nabla g(x, t)|^2 \theta (x)\, dx
& \le \int_\Omega |\nabla g(x, t_0)|^2 \theta (x)\, dx
+ \Big|
\int_t^{t_0} \!\!\! \int_\Omega \partial_s (|\nabla g (x, s)|^2 \theta (x) ) \, dx ds \Big|\\
&\le \frac{2}{T} \int_{\Omega_T} |\nabla g|^2 
+ \int_{\Omega_T}  |\nabla^2 g| |\partial_t g|
+ 2 \int_{\Omega_T}  |\nabla g| |\partial_t g| |\nabla \theta|, 
\endaligned
$$
where we have used an  integration by parts in $x$  for the last step.
By integrating the inequality above in the variable  $t$ over the interval $(0, \rho^2)$, we obtain 
$$
\int_0^{\rho^2} \!\!\!
\int_\Omega
|\nabla g |^2 \theta\, dx dt
\le C \rho
\int_{\Omega_T} 
\Big\{
|\nabla g |^2 +  |\nabla^2 g|^2 + |\partial_t g|^2 \Big\},
$$
where we also used the Cauchy inequality.
This completes the proof.
\end{proof}

\begin{lemma}\label{lemma-6.1}
Let $\Omega$ be a bounded Lipschitz domain in $\mathbb{R}^d$ and $0<T<\infty$.
Let $u_{\e, \lambda} $ be a weak solution of (\ref{DP}) and
$u_{0, \lambda} $ the homogenized problem (\ref{DP-0}).
Let $w_\e$ be defined by (\ref{w-3}).
Then, for any $\psi\in L^2(0, T; H_0^1(\Omega))$,
\begin{equation}\label{6.1}
\aligned
& \Big|\int_0^T \langle  \partial_t w_\e, \psi \rangle_{H^{-1}(\Omega) \times H^1_0(\Omega)}
+\int_{\Omega_T} A_\lambda (x/\e, t/\e^2)\nabla w_\e \cdot \nabla \psi \Big|\\
& \le  C \Big\{ \| u_{0, \lambda} \|_{L^2(0, T; H^2(\Omega))}
+ \|\partial_t u_{0, \lambda}\|_{L^2(\Omega_T)} \Big\}
\Big\{ \delta  \|\nabla \psi \|_{L^2(\Omega_T)}
+ \delta^{1/2}
\| \nabla \psi\|_{L^2(\Omega_{T, 3\delta})}
\Big\},
\endaligned
\end{equation}
where $\delta=(1+\sqrt{\lambda})\e$ and $C$ depends only on $d$,  $\mu$, $\Omega$ and $T$.
\end{lemma}

\begin{proof}
In view of Lemma \ref{lemma-6.0}, 
the case $\lambda=1$ follows from  \cite[Lemma 3.5]{Geng-Shen-2017}.
The case $\lambda\neq 1$ is proved in a similar manner.
Indeed, by (\ref{w-3.1}), the right-hand side of (\ref{6.1}) is bounded by 
$$
\aligned
  & C\int_{\Omega_T} 
|\nabla u_{0, \lambda} -K_\e (\nabla u_{0, \lambda})| |\nabla \psi|
+C \e \int_{\Omega_T} |(\chi^\lambda)^\e | |\nabla K_\e (\nabla u_{0, \lambda})| |\nabla \psi|\\
&\quad  + C\e \int_{\Omega_T} \sum_{k, i, j} | (\phi_{kij}^\lambda)^\e| |\nabla K_\e (\nabla  u_{0, \lambda})| |\nabla \psi|\\
&\quad  + C \e^2 \int_{\Omega_T}
\sum_{k, j} |(\phi_{k (d+1) j}^\lambda)^\e | |\partial_t K_\e (\nabla u_{0, \lambda})| |\nabla \psi|\\
& \quad + C \e \int_{\Omega_T}
\sum_{k, j} |(\nabla \phi_{ k (d+1) j}^\lambda)^\e | |\nabla K_\e (\nabla u_{0, \lambda}| |\nabla \psi|\\
& \quad + C \e^2 \int_{\Omega_T}
\sum_{k, j} 
|(\phi_{k (d+1) j} ^\lambda)^\e | |\nabla^2 K_\e (\nabla u_{0, \lambda})| |\nabla \psi|\\
&=I_1 +I_2+I_3+I_4+I_5+I_6.
\endaligned
$$
The estimates of $I_j$ for $j=1, \dots, 6$ are exactly the same as in the proof of Lemma 3.5 in \cite{Geng-Shen-2017}.
Also see the proof of Lemma  \ref{lemma-3.4} in Section \ref{section-3}.
We  point out that in the cases of $I_4$ and $I_6$,  the estimate
$$
\sup_{(x, t)\in \mathbb{R}^{d+1} }
\left(\fint_{Q_\delta (x, t)} 
| (\phi^\lambda_{k (d+1) j})^\e |^2\right)^{1/2}
\le C (1+\lambda)
$$
is used.
We omit the details.
\end{proof}

The next theorem gives an error estimate for the  two-scale expansion 
\begin{equation}\label{w-6}
\widetilde{w}_\e
(x, t)=u_{\e, \lambda}  -u_{0, \lambda} -\e \chi^\lambda (x/\e, t/\e^2) K_\e (\nabla u_{0, \lambda} )
\end{equation}
 in $L^2(0, T; H^1(\Omega))$.

\begin{thm}\label{theorem-H-1}
Let $\widetilde{w}_\e$ be defined by (\ref{w-6}).
Under the same conditions as in Lemma \ref{lemma-6.1},
we have
\begin{equation}\label{6.1-0}
\|\nabla \widetilde {w}_\e\|_{L^2(\Omega_T)}
\le C \sqrt{\delta}  \, 
\Big\{
\| u_{0, \lambda} \|_{L^2(0, T; H^2(\Omega))}
+ \|\partial_t u_{0, \lambda} \|_{L^2(\Omega_T)}
\Big\},
\end{equation}
where  $\delta =(1+\sqrt{\lambda})\e\le 1$ and
$C$ depends only on $d$, $\mu$, $\Omega$ and $T$.
\end{thm}

\begin{proof}
Let $\psi=w_\e $ in (\ref{6.1}), where $w_\e$ is given by (\ref{w-3}).
Since  $w_\e =0$ on $\partial_p \Omega_T$, 
we see that $\int_0^T \langle \partial_t w_\e, w_\e\rangle \ge 0$.
It follows that $\|\nabla w_\e\|_{L^2(\Omega_T)}$ is bounded by the right-hand side of (\ref{6.1-0}).
It is not hard to show that
$\| \nabla (w_\e -\widetilde{w}_\e)\|_{L^2(\Omega_T)}$ is also bounded by the right-hand side of (\ref{6.1-0}).
This gives  the inequality (\ref{6.1-0}).
\end{proof}

We now move on to the convergence rate of $u_{\e, \lambda} -u_{0, \lambda}$  in $L^2 (\Omega_T)$.

\begin{thm}\label{theorem-L-2}
Suppose $A$ satisfies  (\ref{ellipticity}) and (\ref{periodicity}).
Let $\Omega$ be a bounded $C^{1, 1}$ domain in $\mathbb{R}^d$.
Let $u_{\e, \lambda}$ be a weak solution of (\ref{DP}) and $u_{0, \lambda}$ the solution of the 
homogenized problem (\ref{DP-0}).
Then
\begin{equation}\label{6.3-0}
\| u_{\e, \lambda}  -u_{0, \lambda}\|_{L^2(\Omega_T)}
\le C \delta 
\Big\{
\| u_{0, \lambda} \|_{L^2(0, T; H^2(\Omega))}
+ \|\partial_t u_{0, \lambda}\|_{L^2(\Omega_T)}
 \Big\},
\end{equation}
where  $\delta =(1+\sqrt{\lambda} ) \e$ and 
$C$ depends only on $d$,   $\mu$, $\Omega$ and $T$.
\end{thm}

\begin{proof}
In view of Lemma \ref{lemma-6.0},
this theorem was proved in \cite[Theorem 1.1]{Geng-Shen-2017} for the case $\lambda=1$.
With Lemma \ref{lemma-6.1} at our disposal,  the case $\lambda\neq 1$ follows
by a similar duality argument.
We omit  the details.
\end{proof}

Finally, we study the problem of convergence rates for  the parabolic  operator 
$\partial_t +\mathcal{L}_\e$, where 
$\mathcal{L}_\e =-\text{\rm div} \big (A(x/\e, t/\e^k)\nabla \big)$ and
$0<k<\infty$.
Note that the case $k=2$ is already treated  in Theorems \ref{theorem-H-1} and \ref{theorem-L-2} with $\lambda=1$.

For the case $k\neq 2$, we use the fact that $\mathcal{L}_\e =\mathcal{L}_{\e, \lambda}$ with $\lambda=\e^{k-2}$.
Recall that the homogenized operator for $\partial_t +\mathcal{L}_\e$ is given 
by $\partial_t -\text{\rm div} \big(\widehat{A_\infty}\nabla \big)$ for $0<k<2$, and
by $\partial_t -\text{\rm div} \big(\widehat{A_0}\nabla \big)$ for $2<k<\infty$,
where $\widehat{A_\infty}$ and $\widehat{A_0}$ are defined in (\ref{2.11} ) and (\ref{2.16}), respectively.

\begin{thm}
\label{L-2-2}
Assume $A$ satisfies  (\ref{ellipticity}) and (\ref{periodicity}).
Also assume that $\| \partial_s A \|_\infty \le M$.
Let $0<k<2$.
Let $u_\e$ be the  weak solution of the initial-Dirichlet problem,
\begin{equation}\label{DP-e}
\partial_t  u_\e -\text{\rm div} \big( A(x/\e, t/\e^k)\nabla u_\e \big) =F 
\quad \text{ in } \Omega_T \quad \text{ and } \quad u_\e =f \quad \text{ on } \partial_p \Omega_T,
\end{equation}
where $\Omega$ is a bounded $C^{1, 1}$ domain in $\mathbb{R}^d$ and
$0<T<\infty$.
Let $u_0$ be the solution of the homogenized problem.
Then
\begin{equation}\label{L-k-0}
\| u_\e -u_0\|_{L^2(\Omega_T)}
\le C (\e^{k/2} +\e^{2-k})
\Big\{ \| u_0\|_{L^2(0, T; H^2(\Omega))}
+\|\partial_t u_0\|_{L^2(\Omega_T)} \Big\}
\end{equation}
for $0<\e\le 1$,
where $C$ depends only on $d$, $\mu$,  $\Omega$, $T$, and $M$.

\begin{proof}
Let $\lambda=\e^{k-2}$ and $u_{0, \lambda}$ be the solution of the initial-Dirichlet problem,
\begin{equation}\label{DP-lambda}
\partial_t  u_{0, \lambda}  -\text{\rm div} \big( \widehat{A_\lambda}\nabla u_{0, \lambda} \big) =F 
\quad \text{ in } \Omega_T \quad \text{ and } \quad u_{0, \lambda} =f \quad \text{ on } \partial_p \Omega_T.
\end{equation}
Note that $(1+\sqrt{\lambda}) \e =   \e +\e^{k/2} \le 2 \e^{k/2}$ for $0<\e \le 1$.
It follows by Theorem \ref{theorem-L-2} that
\begin{equation}\label{6-10}
\| u_\e -u_{0, \lambda} \|_{L^2(\Omega_T)}
\le C \e^{k/2}  
\Big\{ \|  u_{0, \lambda} \|_{L^2(0, T; H^2(\Omega))}
+\|\partial_t u_{0, \lambda} \|_{L^2 (\Omega_T)} \Big\}.
\end{equation}

Next, we observe that $u_{0, \lambda} -u_0=0$ on $\partial_p \Omega_T$ and
$$
\partial_t (u_{0, \lambda} -u_0)
-\text{\rm div } \big( \widehat{A_\lambda } \nabla (u_{0, \lambda} -u_0) \big)
=\text{\rm div} \big( (\widehat{A_\lambda} -\widehat{A_\infty})  \nabla u_0 \big)
$$
in $\Omega_T$.
Since $\Omega$ is $C^{1,1}$,
it follows by the standard regularity estimates for parabolic systems with constant coefficients that
$$
\aligned
& \| \partial_t (u_0 -u_{0, \lambda}) \|_{L^2(\Omega_T)}
+ \| u_0- u_{0, \lambda}\|_{L^2(0, T; H^2(\Omega))}\\
&\qquad\qquad\qquad\qquad
 \le C |\widehat{A_\lambda} -\widehat{A_\infty} | \| \nabla^2 u_0\|_{L^2 (\Omega_T)}\\
&\qquad\qquad\qquad\qquad
\le C \lambda^{-1}  \|\partial_s A \|_\infty 
\|\nabla^2 u_0\|_{L^2(\Omega_T)},
\endaligned
$$
where we have used (\ref{2.21}) for the last step.
This, together with (\ref{6-10}), yields the estimate (\ref{L-k-0}).
\end{proof}
\end{thm}

The next theorem treats the case $2<k<\infty$.

\begin{thm}
\label{L-2-3}
Assume  $A$ satisfies  (\ref{ellipticity}) and (\ref{periodicity}).
Also assume that $\|\nabla^2  A \|_\infty \le M$.
Let $2<k<\infty$.
Let $u_\e$ be the  weak solution of the initial-Dirichlet problem (\ref{DP-e}),
where $\Omega$ is a bounded $C^{1, 1}$ domain in $\mathbb{R}^d$ and
$0<T<\infty$.
Let $u_0$ be the solution of the homogenized problem.
Then
\begin{equation}\label{L-k-1}
\| u_\e -u_0\|_{L^2(\Omega_T)}
\le C (\e  +\e^{k-2})
\Big\{ \| u_0\|_{L^2(0, T; H^2(\Omega))}
+\|\partial_t u_0\|_{L^2(\Omega_T)} \Big\}
\end{equation}
for $0<\e< 1$,
where $C$ depends only on $d$, $\mu$, $\Omega$, $T$, and $M$.
\end{thm}

\begin{proof}
The proof is similar to that of Theorem \ref{L-2-2}.
The only modification is that in the place of (\ref{2.22}), we use the estimate (\ref{2.31}) to 
bound $|\widehat{A_\lambda} -\widehat{A_0|}$.
Also, note that $\|\nabla A\|_\infty$ may be bounded by  a constant depending on $\mu$ and $M$.
We omit the details.
\end{proof}

\begin{proof}[\bf Proof of Theorem \ref{m-theorem-2}]
Let  $0<\e<1$.
Note that    $\e^{2-k} \le \e^{k/2}$ if $0< k\le 4/3$, and $\e^{2-k}\le \e^{k/2}$ if $4/3<k< 2$.
Also, $\e \le \e^{k-2}$ if $2<k< 3$, and $\e^{k-2}\le \e$ if $k\ge 3$.
Thus, by Theorems \ref{L-2-2} and \ref{L-2-3},
$$
\aligned
 & \| u_\e -u_0\|_{L^2(\Omega_T)}\\
 & 
\le  C \Big\{ \| u_0\|_{L^2(0, T; H^2(\Omega))}
+\|\partial_t u_0\|_{L^2(\Omega_T)} \Big\}
\cdot 
\left\{
\aligned
&  \e^{k/2} 
 &  \quad &  \text{ if } 0<k\le 4/3,\\
&\e^{2-k} 
&  \quad &  \text{ if } 4/3< k<  2,\\
& \e^{k-2} & \quad & \text{ if } 2<k<3,\\
&  \e
& \quad &  \text{ if } k=2 \text{ or } 3\le  k<\infty .\\
\endaligned
\right.
\endaligned
$$
\end{proof}

\begin{remark}
The results on convergence rates in Theorems \ref{L-2-2} and \ref{L-2-3} also hold for initial-Neumann problems.
The proof is almost identical to the case of the  initial-Dirichlet problem.
See \cite{Geng-Shen-2017} for the case $k=2$.
\end{remark}

Using Theorem \ref{theorem-H-1} we may obtain an error estimate in $L^2(0, T; H^1(\Omega))$ for a  two-scale expansion 
for $\partial_t + \mathcal{L}_\e$ in (\ref{operator-0}) 
in terms of its  own correctors.
The case $k=2$ is contained in Theorem \ref{theorem-H-1} with $\lambda=1$.
For $k\neq 2$, we let
\begin{equation}\label{6-30}
v_\e=
\left\{
\aligned
& u_\e -u_0 -\e \chi^\infty (x/\e, t/\e^k) \widetilde{K}_\e (\nabla u_0)& \quad & \text{ if } 0<k<2,\\
& u_\e  -u_0 -\e \chi^0 (x/ \e, t/\e^k) \widetilde{K}_\e (\nabla u_0) & \quad  & \text{ if } 2<k<\infty.
\endaligned
\right.
\end{equation}
In (\ref{6-30}), $\chi^\infty$ and $\chi^0$ are the correctors 
defined by (\ref{2.8}) and (\ref{2.14}), respectively, for $\partial_t + \mathcal{L}_\e$.
Since they satisfy the estimates (\ref{2.10}) and (\ref{2.15}), 
only smoothing in the space variable is needed for the operator $\widetilde{K}_\e$.
More precisely, we let $\widetilde{K}_\e (f)= S^1_\delta  (\eta_\delta f)$,
where 
$$
S^1_\delta (f) (x, t) = \int_{\mathbb{R}^d}  f(x-y, t)  \delta^{-d} \theta_1 (y/\delta)\, dy,
$$ 
$\delta=\e +  \e^{k/2}$,
and the cut-off functions $\eta_\delta$ is the same as in $K_\e$.

\begin{thm}\label{H-1-0}
Suppose that $A$ and $\Omega$  satisfy the same conditions as in Theorem \ref{L-2-2}.
Let $u_\e$ be the weak solution of (\ref{DP-e}) and $u_0$ the homogenized solution.
Let $v_\e$ be given by (\ref{6-30}).
Then 
\begin{equation}\label{6-40}
\aligned
 & \| \nabla v_\e \|_{L^2(\Omega_T)}\\
 & 
\le  C \Big\{ \| u_0\|_{L^2(0, T; H^2(\Omega))}
+\|\partial_t u_0\|_{L^2(\Omega_T)} \Big\}
\cdot 
\left\{
\aligned
&  \e^{k/4} 
 &  \quad &  \text{ if }\  0<k\le 8/5,\\
&\e^{2-k} 
&  \quad &  \text{ if }\  8/5< k<  2,\\
& \e^{k-2} & \quad & \text{ if }\  2<k<5/2,\\
&  \e^{1/2}
& \quad &  \text{ if }  \ 5/2\le  k<\infty .\\
\endaligned
\right.
\endaligned
\end{equation}
\end{thm}

\begin{proof}
The proof uses Theorem \ref{theorem-H-1} and the estimates of $u_{0, \lambda} -u_0$ in 
the proof of Theorems \ref{L-2-2} and \ref{L-2-3}, where $u_{0, \lambda}$ is the solution of
(\ref{DP-lambda}) with $\lambda=\e^{2-k}$.

Let $\lambda =\e^{k-2}$.
Suppose $0<k<2$.
In view of (\ref{6.1-0}) it suffices to bound 
$$
I=\| \nabla \Big\{
\e \chi^\lambda (x/ \e, t/\e^2) K_\e (\nabla u_{0, \lambda})
-\e \chi^\infty (x/\e, t/\e^k ) 
\widetilde{K}_\e (\nabla u_0)\Big\} \|_{L^2(\Omega_T)}.
$$
Note that
$$
\aligned
I &\le 
\|  \big(\nabla \chi^\lambda (x/\e, t/\e^2) -\nabla\chi^\infty (x/\e, t/\e^k)\big)
K_\e (\nabla u_{0, \lambda})\|_{L^2(\Omega_T)}\\
& \qquad + \|\nabla \chi^\infty (x/\e, t/\e^k) \big( K_\e (\nabla u_{0, \lambda})
-\widetilde{K}_\e (\nabla u_0) \big)\|_{L^2(\Omega_T)}\\
&\qquad  +\e \|\chi^\lambda (x/\e, t/\e^2) \nabla K_\e (\nabla u_{0, \lambda}) \|_{L^2(\Omega_T)}\\
&\qquad +\e \| \chi^\infty (x/\e, t/\e^k) \nabla \widetilde{K}_\e (\nabla u_0)\|_{L^2(\Omega_T)}\\
&=I_1 +I_2 +I_3 +I_4.
\endaligned
$$
To bound $I_1$, we use the inequality (\ref{S-1}).
This gives
\begin{equation}\label{I-1}
\aligned
I_1 &  \le C \sup_{(y, s)\in \mathbb{R}^{d+1}}
\left(\fint_{Q_{\delta}}
|\nabla \chi^\lambda (y/\e, t/\e^2) -\nabla \chi^\infty(x/\e, t/\e^k )|^2 \, dy ds \right)^{1/2}
\| \nabla u_{0, \lambda} \|_{L^2(\Omega_T)}\\
&\le  C \lambda^{-1} \|\partial_s A\|_\infty \| \nabla u_{0, \lambda}\|_{L^2 (\Omega_T)}\\
&\le C \e^{2-k} \|\partial_s A\|_\infty \|\nabla u_0\|_{L^2(\Omega_T)},
\endaligned
\end{equation}
where we have used (\ref{Q-e-5}) for the second inequality.
To estimate $I_2$, we assume that the function $\theta_1$ is chosen so that
$\theta_1 =\theta_{11} * \theta_{11}$, where $\theta_{11} \in C^\infty_0 (B(0, 1))$,
$\theta_{11} \ge 0$ and $\int_{\mathbb{R}^d} \theta_{11} =1$.
This allows us to write $S_\delta^1 = S_\delta^{11} \circ S_\delta^{11}$,
where $S_\delta^{11} (f)=f * (\theta_{11})_\delta$.
As a result, we obtain 
$$
\aligned
I_2  &\le C \| S_\delta^2\big[ S_\delta^{11}(\eta_\delta \nabla u_0)\big] -S^{11}_\delta (\eta_\delta \nabla u_0)\|_{L^2(\Omega_T)}
\\
&\le C\delta^2 \| \partial_t S_\delta^{11} (\eta_\delta \nabla u_0)\|_{L^2(\Omega_T)}\\
&= C \delta^2 
\ \| S_\delta^{11}
 \big\{  (\partial_t  \eta_\delta) (\nabla u_0)
+\nabla (\eta_\delta \partial_t u_0)
-(\nabla \eta_\delta) \partial_t u_0 \big\} \|_{L^2(\Omega_T)}\\
&\le C \delta^{1/2}
\Big\{ \|\nabla u_0\|_{L^2(\Omega_T)}
+\|\nabla^2 u_0\|_{L^2(\Omega_T)}
+ \|\partial_t u_0\|_{L^2(\Omega_T)} \Big\}.
\endaligned
$$
It is not hard to see that
$$
\aligned
I_3 +I_4
 & \le C \e \Big\{
\|\nabla (\eta_\delta \nabla u_{0, \lambda}) \|_{L^2(\Omega_T)}
+ \|\nabla (\eta_\delta \nabla u_{0}) \|_{L^2(\Omega_T)}\Big\}\\
& \le C \delta^{1/2}
\Big\{ \|\nabla u_0\|_{L^2(\Omega_T)}
+\|\nabla^2 u_0\|_{L^2\Omega_T)} \Big\}.
\endaligned
$$
In summary we have proved that
\begin{equation}\label{H-20}
\|\nabla v_\e\|_{L^2(\Omega_T)}
\le C \big\{
 \e^{k/4} +\e^{2-k} \big\}
\big\{ \|u_0\|_{L^2(0, T; H^1(\Omega))}
+\|\partial_t u_0\|_{L^2(\Omega_T)} \big\}
\end{equation}
for $0<k<2$.
A similar argument gives
\begin{equation}\label{H-21}
\|\nabla v_\e\|_{L^2(\Omega_T)}
\le C \big\{
 \e^{1/2} +\e^{k-2} \big\}
\big\{ \|u_0\|_{L^2(0, T; H^1(\Omega))}
+\|\partial_t u_0\|_{L^2(\Omega_T)} \big\}
\end{equation}
for $2<k<\infty$.
The error estimate (\ref{6-40}) follows readily from (\ref{H-20}) and (\ref{H-21}).
\end{proof}


 \bibliographystyle{amsplain}
 
\bibliography{Geng-Shen-2019.bbl}

\providecommand{\bysame}{\leavevmode\hbox to3em{\hrulefill}\thinspace}
\providecommand{\MR}{\relax\ifhmode\unskip\space\fi MR }
\providecommand{\MRhref}[2]{%
  \href{http://www.ams.org/mathscinet-getitem?mr=#1}{#2}
}
\providecommand{\href}[2]{#2}
\begin{thebibliography}{10}

\bibitem{Allaire-1996}
G.~Allaire and M.~Briane, \emph{Multiscale convergence and reiterated
  homogenisation}, Proc. Roy. Soc. Edinburgh Sect. A \textbf{126} (1996),
  no.~2, 297--342.

\bibitem{Armstrong-2018}
S.~N. Armstrong, A.~Bordas, and J.-C. Mourrat, \emph{Quantitative stochastic
  homogenization and regularity theory of parabolic equations}, Anal. PDE
  \textbf{11} (2018), no.~8, 1945--2014.

\bibitem{AKM-2016}
S.~N. Armstrong, T.~Kuusi, and J.-C. Mourrat, \emph{Mesoscopic higher
  regularity and subadditivity in elliptic homogenization}, Comm. Math. Phys.
  \textbf{347} (2016), no.~2, 315--361.

\bibitem{Armstrong-2017}
\bysame, \emph{The additive structure of elliptic homogenization}, Invent.
  Math. \textbf{208} (2017), no.~3, 999--1154.

\bibitem{Armstrong-book}
\bysame, \emph{Quantitative stochastic homogenization and large-scale
  regularity}, book preprint, 2017.

\bibitem{AM-2016}
S.~N. Armstrong and J.-C. Mourrat, \emph{Lipschitz regularity for elliptic
  equations with random coefficients}, Arch. Ration. Mech. Anal. \textbf{219}
  (2016), no.~1, 255--348.

\bibitem{A-Shen-2016}
S.~N. Armstrong and Z.~Shen, \emph{Lipschitz estimates in almost-periodic
  homogenization}, Comm. Pure Appl. Math. \textbf{69} (2016), no.~10,
  1882--1923.

\bibitem{Armstrong-2016}
S.~N. Armstrong and C.~Smart, \emph{Quantitative stochastic homogenization of
  convex integral functionals}, Ann. Sci. \'Ec. Norm. Sup\'er. (4) \textbf{49}
  (2016), no.~2, 423--481.

\bibitem{AL-1987}
M.~Avellaneda and F.~Lin, \emph{Compactness methods in the theory of
  homogenization}, Comm. Pure Appl. Math. \textbf{40} (1987), no.~6, 803--847.

\bibitem{BLP-2011}
A.~Bensoussan, J.-L. Lions, and G.~Papanicolaou, \emph{Asymptotic analysis for
  periodic structures}, AMS Chelsea Publishing, Providence, RI, 2011, Corrected
  reprint of the 1978 original.

\bibitem{Otto-2016}
J.~Fischer and F.~Otto, \emph{A higher-order large-scale regularity theory for
  random elliptic operators}, Comm. Partial Differential Equations \textbf{41}
  (2016), no.~7, 1108--1148.

\bibitem{Floden-2006}
L.~Flod\'{e}n and M.~Olsson, \emph{Reiterated homogenization of some linear and
  nonlinear monotone parabolic operators}, Can. Appl. Math. Q. \textbf{14}
  (2006), no.~2, 149--183.

\bibitem{Geng-Shen-2015}
J.~Geng and Z.~Shen, \emph{Uniform regularity estimates in parabolic
  homogenization}, Indiana Univ. Math. J. \textbf{64} (2015), no.~3, 697--733.

\bibitem{Geng-Shen-2017}
\bysame, \emph{Convergence rates in parabolic homogenization with
  time-dependent periodic coefficients}, J. Funct. Anal. \textbf{272} (2017),
  no.~5, 2092--2113.

\bibitem{Geng-Shen-2018}
\bysame, \emph{Asymptotic expansions of fundamental solutions in parabolic
  homogenization}, Anal. PDE ({To appear}).

\bibitem{Gloria-2015}
A.~Gloria, S.~Neukamm, and F.~Otto, \emph{Quantification of ergodicity in
  stochastic homogenization: optimal bounds via spectral gap on {G}lauber
  dynamics}, Invent. Math. \textbf{199} (2015), no.~2, 455--515.

\bibitem{Gloria-2017}
A.~Gloria and F.~Otto, \emph{Quantitative results on the corrector equation in
  stochastic homogenization}, J. Eur. Math. Soc. (JEMS) \textbf{19} (2017),
  no.~11, 3489--3548.

\bibitem{Holmbom}
A.~Holmbom, N.~Svanstedt, and N.~Wellander, \emph{Multiscale convergence and
  reiterated homogenization of parabolic problems}, Appl. Math. \textbf{50}
  (2005), no.~2, 131--151.

\bibitem{KLS-2013-N}
C.~Kenig, F.~Lin, and Z.~Shen, \emph{Homogenization of elliptic systems with
  {N}eumann boundary conditions}, J. Amer. Math. Soc. \textbf{26} (2013),
  no.~4, 901--937.

\bibitem{Suslina-M-2015}
Yu.~M. Meshkova and T.~A. Suslina, \emph{Homogenization of solutions of initial
  boundary value problems for parabolic systems}, Funct. Anal. Appl.
  \textbf{49} (2015), no.~1, 72--76, Translation of Funktsional. Anal. i
  Prilozhen. {{\bf{4}}9} (2015), no. 1, 88--93.

\bibitem{Niu-Xu-2019}
W.~Niu and Y.~Xu, \emph{A refined convergence result in homogenization of
  second order parabolic systems}, Preprint.

\bibitem{Niu-Xu-2018}
\bysame, \emph{Convergence rates in homogenization of higher-order parabolic
  systems}, Discrete Contin. Dyn. Syst. \textbf{38} (2018), no.~8, 4203--4229.

\bibitem{Pastukhova-2009}
S.~E. Pastukhova and R.~N. Tikhomirov, \emph{Estimates for locally periodic and
  reiterated homogenization: parabolic equations}, Dokl. Akad. Nauk
  \textbf{428} (2009), no.~2, 166--170.

\bibitem{Shen-book}
Z.~Shen, \emph{{Periodic Homogenization of Elliptic Systems}}, Operator Theory:
  Advances and Applications, vol. 269, Birkh\"{a}user/Springer, Cham, 2018,
  Advances in Partial Differential Equations (Basel).

\bibitem{Suslina-2013}
T.~A. Suslina, \emph{Homogenization of the {D}irichlet problem for elliptic
  systems: {$L_2$}-operator error estimates}, Mathematika \textbf{59} (2013),
  no.~2, 463--476.

\bibitem{Woukeng-2010}
J.-L. Woukeng, \emph{{$\Sigma$}-convergence and reiterated homogenization of
  nonlinear parabolic operators}, Commun. Pure Appl. Anal. \textbf{9} (2010),
  no.~6, 1753--1789.

\bibitem{Xu-Zhou}
Q.~Xu and S.~Zhou, \emph{Quantitative estimates in homogenization of parabolic
  systems of elasticity in {L}ipschitz cylinders}, Preprint.

\end{thebibliography}

\bigskip

\begin{flushleft}
Jun Geng,
School of Mathematics and Statistics,
Lanzhou University,
Lanzhou, P.R. China.

E-mail:gengjun@lzu.edu.cn
\end{flushleft}

\begin{flushleft}
Zhongwei Shen,
Department of Mathematics,
University of Kentucky,
Lexington, Kentucky 40506,
USA.

E-mail: zshen2@uky.edu
\end{flushleft}

\bigskip

\medskip

\end{document}